\documentclass{amsart}
\usepackage{amssymb, enumerate}
\usepackage[colorlinks, linkcolor=blue, citecolor=red]{hyperref}

\newtheorem{theorem}{Theorem}[section]
\newtheorem{mainresult}{Main Theorem}
\newtheorem{lemma}[theorem]{Lemma}

\theoremstyle{plain}
\newtheorem{claim}[theorem]{Claim}
\newtheorem{corollary}[theorem]{Corollary}
\newtheorem{prop}[theorem]{Proposition}

\theoremstyle{definition}
\newtheorem{definition}[theorem]{Definition}

\theoremstyle{remark}
\newtheorem{remark}[theorem]{Remark}

\numberwithin{equation}{section}

\newcommand{\CC}{\mathbb{C}}
\newcommand{\GG}{\mathbb{G}}
\newcommand{\M}{\overline{M}}
\newcommand{\HH}{\mathrm{H}}
\newcommand{\PP}{\mathbb{P}}
\renewcommand{\O}{\mathcal{O}}
\newcommand{\QQ}{\mathbb{Q}}
\newcommand{\Proj}{\operatorname{Proj}}
\newcommand{\SL}{\operatorname{SL}}
\newcommand{\PGL}{\operatorname{PGL}}
\newcommand{\gitq}{/\hspace{-0.20pc}/}
\newcommand{\ra}{\rightarrow}
\newcommand{\co}{\colon\thinspace} 

\newcommand{\Sym}{\operatorname{Sym}}
\newcommand{\rank}{\operatorname{rank}}
\newcommand\Mg[1]{\overline{\mathcal{M}}_{#1}}
\newcommand{\MX}{\M^{\, G}}
\newcommand{\Trig}{\operatorname{Trig}}

\newcommand\dra{\dashrightarrow}

\newcommand\nb{\nobreakdash}

\DeclareMathOperator{\Gr}{\mathbb{G}}

\begin{document}

\title{Stability of genus five canonical curves}

\author{Maksym Fedorchuk}
\address[Fedorchuk]{Department of Mathematics, Boston College, 
140 Commonwealth Avenue,
Chestnut Hill, MA 02467}
\email{maksym.fedorchuk@bc.edu}

\author{David Ishii Smyth}
\address[Smyth]{Department of Mathematics, Harvard University,
1 Oxford Street,
Cambridge, MA 01238}
\curraddr{Centre for Mathematics and its Applications, Mathematical Sciences Institute,
    Australian National University, Canberra, ACT 0200, Australia}
\email{david.smyth@anu.edu.au}
\thanks{The first author was partially supported by NSF grant DMS-1259226. 
The second author was partially supported by NSF grant DMS-0901095.}
\date{}

\dedicatory{Dedicated to our advisor, Joe Harris, on his sixtieth birthday.}


\begin{abstract}
We analyze GIT stability of nets of quadrics in $\PP^4$ up to projective equivalence. 
Since a general net of quadrics defines a canonically embedded smooth curve of genus $5$, 
the resulting quotient $\MX:=\Gr(3,15)^{ss} \gitq \SL(5)$ gives a birational model of $\M_5$. 
We study the geometry of the associated contraction $f\co \M_{5} \dashrightarrow \MX$, 
and prove that $f$ is the final step in the log minimal model program for $\M_{5}$.
\end{abstract}

\maketitle

\section{Introduction}

A canonically embedded non-hyperelliptic, non-trigonal smooth curve of genus $5$ 
is a complete intersection of $3$ quadrics in $\PP^{4}$ \cite[Ch.V]{ACGH}. 
Thus, the Grassmannian of nets in $\PP\HH^0(\PP^4, \O_{\PP^4}(2))\simeq \PP^{14}$ 
gives a natural compactification of the open Hilbert scheme of non-hyperelliptic, 
non-trigonal smooth canonical curves of genus $5$, and the corresponding GIT quotient 
\[
\MX:=\Gr(3,15)^{ss} \gitq \SL(5)
\] 
is a projective birational model of $\M_5$. 
In this paper, we study the geometry of $\MX$ and show that the natural birational contraction 
$f\co \M_5 \dashrightarrow \MX$ represents the final stage of the log minimal model program for $\M_{5}$.

The main portion of the paper is devoted to a GIT stability analysis of nets of quadrics in $\PP^4$. 
The GIT stability analysis for {\em pencils} of quadrics in $\PP^{4}$ appears 
in \cite{miranda}, where it is shown that a pencil of quadrics
in $\PP^4$ is semi-stable if and only if the associated discriminant binary quintic is non-zero
and has no triple roots, and in \cite{mabuchi-mukai}. More generally a pencil of quadrics in $\PP^{n}$ 
is semi-stable if and only if the associated discriminant binary $(n+1)$-form is non-zero and is
GIT-semi-stable with respect to the natural $\SL(2)$-action \cite[Theorem 5]{miranda}. 
The GIT analysis for nets of quadrics turns out to be more involved. 
In particular, as Remark \ref{R:triple-conic} shows,
there is no natural
correspondence between $\SL(5)$-stability of a net and $\SL(3)$-stability of the associated
discriminant quintic curve. 

We prove that a semi-stable net defines 
a locally planar curve of genus $5$ embedded  in $\PP^4$ by its dualizing sheaf, and give a description of the 
singularities occurring on such curves.

\begin{mainresult}\label{MT1}
A net is semi-stable if and only if it defines a locally planar genus $5$ curve 
satisfying one of the following conditions:
\begin{enumerate}
\item $C$ is a reduced quadric section of a smooth quartic del Pezzo in $\PP^4$,
but $C$ is \emph{not} one of the following:
\begin{enumerate}
\item A union of an elliptic quartic curve and two conics meeting in a pair of triple points.
\item A union of two elliptic quartics meeting along an $A_5$ and an $A_1$ singularities.
\item A curve with a $4$-fold point with two lines as its two branches.
\item A union of two tangent conics and an elliptic quartic 
meeting the conics in a $D_6$ singularity and two nodes.
\item A curve with a $D_5$ singularity such that the hyperelliptic involution of the normalization exchanges the
points lying over the $D_5$ singularity. 
\item $C$ contains a conic meeting the residual genus $2$ component in an $A_7$ singularity 
and the attaching point of the genus $2$ component is a Weierstrass point.
\item A degeneration of one of the six curves listed above.
\end{enumerate}
\item $C$ is non-reduced and it degenerates isotrivially to one of the following curves: 
\begin{enumerate}
\item The balanced ribbon, defined by 
\begin{equation}\label{E:ribbon}
(ac-b^2, ae-2bd+c^2, ce-d^2).
\end{equation}
\item A double twisted cubic meeting the residual conic in two points, defined by
\begin{equation}\label{E:dtc}
(ad-bc, ae-c^2+L^2, be-cd),
\end{equation}
where $L$ is a general linear form.

\item A double conic meeting the residual rational normal quartic in three points, defined by
\begin{equation}\label{E:double-conic}
(ad-bc, ae-c^2+bL_1+dL_2, be-cd), 
\end{equation}
where $L_1$ and $L_2$ are general linear forms. In particular, 
we have a semi-stable triple conic with two lines
\begin{equation}\label{E:triple-conic}
(ad-bc, ae+bd-c^2, be-cd).
\end{equation}

\item {Two double lines joined by two conics}, defined by 
\begin{equation}\label{E:double-line}
(ad, ae+bd-c^2, be).
\end{equation}
\end{enumerate}
\end{enumerate}
\end{mainresult}

We should make a comment on the shortcomings of this result. 
While this theorem gives in principle 
a complete characterization of the singularities arising on curves in $\MX$, 
it does not give a satisfactory description of the functor represented by $\MX$. 
The difficulty is that a complete characterization of the functor of semi-stable curves necessarily involves 
the global geometry of the curves in question in a way that defies 
uniform description. For example, $A_1$ and $A_5$ singularities are generally 
allowed, except in the unique case when two 
elliptic quartics meet in $A_1$ and $A_5$ singularities. 
Similarly, $D_5$ singularities are generally allowed, except in the case when the 
hyperelliptic involution on the normalization
exchanges points lying over the singularity.

As a by-product of our GIT analysis, we obtain a good understanding 
of the geometry of the birational map $f\co \M_5 \dra \MX$. 
To state our first result in this direction, let us define the $A_5^{\{1\}}$-locus to be the locus of curves in $\MX$ 
which can be expressed as the union of a genus 3 curve and a smooth rational curve meeting along an 
$A_5$ singularity. The significance of these curves lies in the fact that their stable limits 
are precisely curves in $\Delta_2\subset \M_5$ 
with a genus $2$ component attached at a non-Weierstrass point.
 Our main results regarding the birational geometry of $f$
can now be summarized in the following theorem.
\begin{mainresult}\label{MT2}
The birational map $f\co \M_5 \dra \MX$ is a rational contraction, contracting the following divisors:
\begin{enumerate}
\item $f$ contracts $\Delta_1$ and exhibits the generic point of $\Delta_1$ 
as a fibration over the $A_2$-locus in $\MX$.
\item $f$ contracts $\Delta_2$ and exhibits the generic point of $\Delta_2$ 
as a fibration over the $A_{5}^{\{1\}}$-locus in $\MX$.
\item $f$ contracts the trigonal divisor $\Trig_5 \subset \M_5$ to the single point 
given by Equation \eqref{E:triple-conic}.
\end{enumerate}
\end{mainresult}
In addition, $f$ flips various geometrically significant loci in the boundary of $\M_5$ to associated 
equisingular strata in $\MX$, as summarized in Table \ref{table-replacement}. 
Detailed proofs of the assertions made in this table would take us rather far afield into the intricacies of stable 
reduction; thus, we leave these assertions without proof and merely offer the table as a guide to future 
exploration of $f$.
We refer the reader to \cite{HarMor} for a beautiful introduction to stable reduction, 
\cite{hassett-stable} for the results concerning stable reduction of planar curve singularities, 
and the recent survey \cite{casalaina-martin} for an in-depth guide to stable reduction.

\begin{table}[htb]
\caption{Conjectural outline of the log MMP for $\M_{5}$.}
\label{table-replacement}
\renewcommand{\arraystretch}{1.3}
\begin{tabular}{| c| c | c |}
\hline
$\alpha$ &Singularity Type    & Locus in $\M_5(\alpha+\epsilon)$    \\
\hline
\hline
$9/11$ &$A_2$             & \small{elliptic tails attached nodally}\\
\hline
$7/10$ &$A_3$            & \small{elliptic bridges attached nodally } \\
\hline
$2/3$ &$A_4$                & \small{genus $2$ tails attached nodally at a Weierstrass point}  \\
\hline
$19/29$ &$A_5^{\{1\}}$            & \small{genus $2$ tails attached nodally}\\
\hline
$19/29$ &$A_{3/4}$                & \small{genus $2$ tails attached tacnodally at a Weierstrass point}\\
\hline
$12/19$ &$A_{3/5}$                 & \small{genus $2$ tails attached tacnodally}\\       
\hline
$17/28$ &$A_5$               & \small{genus $2$ bridges attached nodally at
conjugate points}         \\
\hline
$5/9$&$D_4$   &    \small{elliptic triboroughs attached nodally}   \\
\hline
$5/9$&$D_5$    & \small{genus $2$ bridges attached nodally at a
Weierstrass and free point}\\
\hline
$5/9$ &$D_6^{\{1,2\}}$   & \small{genus $2$ bridges attached
nodally at two free points}\\
&double lines   & \\
\hline
$25/44$ &$A_{10}$, $A_{11}$ & \small{hyperelliptic curves} \\
&ribbons&  \\
\hline
$1/2$ &double twisted cubics& \small{irreducible nodal curves with hyperelliptic normalization} \\
&$D_8^{\{1,2\}}$ & \\
\hline
$14/33$&\small{$4$-fold point}  &    \small{genus $3$ triboroughs}   \\
\hline
$14/33$ &triple conics& \small{trigonal curves} \\
\hline
\end{tabular}
\medskip
\end{table}
Finally, we study the contraction $f\co \M_{5} \dra \MX$ from the perspective of the minimal model program. Recall that
\[
\M_{5}(\alpha):=\Proj \bigoplus_{m \geq 0} \HH^0(\Mg{5}, \lfloor m(K_{\Mg{5}}+\alpha \delta) \rfloor), \quad \alpha \in [0,1]\cap \QQ,
\]
and that as $\alpha$ decreases from 1, the corresponding birational models constitute the log minimal model 
program for $\M_{5}$ \cite{hassett_genus2, hassett-hyeon_contraction, hassett-flip}. 
Our final result interprets the contraction $f\co \M_{5} \dra \MX$ as the final step of this program.
\begin{mainresult}\label{T:moving-slope} \hfill

(1) The moving slope of $\M_5$ is $33/4$, realized by the divisor
\[
f^*\O(1) \sim 33\lambda-4\delta_0-15\delta_1-21\delta_2,
\]
where $\sim$ denotes the numerical proportionality.

(2) There is a natural isomorphism $\MX \simeq \M_5(\alpha)$, for all $\alpha \in (3/8, 14/33] \cap \QQ$, 
identifying $f\co \M_5 \dra \MX$ with the final step of the log MMP for $\M_5$. In particular, 
$\M_5(3/8)$ is a point.
\end{mainresult}
In Table \ref{table-replacement}, 
we have listed the $\alpha$-invariants, as defined in \cite{afs-modularity},  
of some singularities appearing on curves parameterized by $\M_5(14/33)$
in order to indicate the anticipated threshold values of $\alpha$ 
at which the transformations should occur in the course of the
log MMP for $\M_5$. 
The reader should also refer to \cite{afs-modularity} for the definition of the notations $A_5^{\{1\}}$, 
$D_{6}^{\{1,2\}}$, $A_{3/4}$, $A_{3/5}$.

Note that while we now have a nearly complete description of the log MMP for $\M_{4}$ 
\cite{hyeon-lee_genus4, fedorchuk-genus-4, CMJL-genus-4},
we have no construction of the intermediate models $\M_{5}(\alpha)$ for $14/33<\alpha\leq 2/3$. 
Table \ref{table-replacement} gives a rather ominous indication of the potential complexity of this task. 

Let us now give a roadmap of the paper. In Section \ref{S:git-analysis}, 
we describe GIT-unstable nets of quadrics. 
We first describe a complete, finite set $\{\rho_i\}_{i=1}^{12}$ of destabilizing one-parameter subgroups 
(Theorem \ref{T:subgroups}), and then provide geometric descriptions of the nets of quadrics destabilized 
by $\rho_i$ for each $1\leq i \leq 12$ (Theorem \ref{T:Description}). 
In Section \ref{S:semi-stable}, we combine a geometric study of quartic surfaces
in $\PP^4$ with Theorem \ref{T:Description} to obtain a positive description of semi-stable nets of quadrics
(Theorems \ref{T:reduced-semi-stable} and \ref{T:non-reduced-semi-stable}). Finally, in Section \ref{S:proofs}
we use our semi-stability results to give proofs of 
Main Theorems \ref{MT1}, \ref{MT2}, \ref{T:moving-slope}.

\subsection*{Acknowledgements}
We are profoundly grateful to Joe Harris for introducing us to the subject of algebraic geometry, 
for generously sharing his love of the discipline and his unique creative style. 
Our understanding and appreciation of Geometric Invariant Theory also owes a great deal 
to the papers and lectures of Brendan Hassett, David Hyeon, and Ian Morrison. 
We wish to thank each of these individuals, as well as Jarod Alper, Anand Deopurkar, David Jensen, 
Radu Laza, and David Swinarski, for numerous conversations and suggestions 
relating to the contents of this paper.

\section{GIT analysis}\label{S:git-analysis}

\subsection{GIT preliminaries}
Set $V:=\HH^0\bigl(\PP^4, \O(1)\bigr)$ and let $W:=\HH^0\bigl(\PP^4, \O(2)\bigr)\simeq \Sym^2 V$ 
be the vector space of quadratic forms. 
To a net of quadrics $\Lambda=\bigl(Q_1, Q_2, Q_3\bigr)$ in $W$, we associate its {\em Hilbert point}
\[
[\Lambda]:=[Q_1\wedge Q_2 \wedge Q_3] \in \GG\bigl(3, W\bigr) \subset \PP \bigwedge^{3} W.
\]
We denote by $\widetilde{[\Lambda]}:=Q_1\wedge Q_2 \wedge Q_3$ 
a lift of $[\Lambda]$ to $\bigwedge^{3} W$. 

Recall that $\Lambda$ is said to be \emph{semi-stable} if $0 \notin \SL(5) \cdot \widetilde{[\Lambda]}$, 
and \emph{stable} if in addition $\SL(5) \cdot \widetilde{[\Lambda]}$ is closed. 
GIT gives a projective quotient $\GG(3,W)^{ss} \gitq \SL(5)$, 
where $\GG(3,W)^{ss} \subset \GG\bigl(3,W\bigr)$ is the open locus of semi-stable nets, and the main 
objective of this paper is to give a geometric description of $\GG(3,W)^{ss}$.

The standard tool for such analysis is the Hilbert-Mumford numerical criterion 
\cite[Theorem 2.1]{GIT}. 
In our situation, the statement of the numerical criterion may be formulated as follows: 
Let $\rho\co \CC^* \rightarrow \SL(5)$ be a 
one-parameter subgroup (1-PS), acting diagonally on a basis $\{a,b,c,d,e\}$ of $V$
with weights $\{\bar{a}, \bar{b}, \bar{c}, \bar{d}, \bar{e}\}$, satisfying:
\begin{itemize}
\item $\bar{a}+\bar{b}+\bar{c}+\bar{d}+\bar{e}=0$,
\item $\bar{a} \geq \bar{b} \geq \bar{c} \geq  \bar{d} \geq \bar{e}$,
\item Not all weights $\{\bar{a}, \bar{b}, \bar{c}, \bar{d}, \bar{e}\}$ are $0$.
\end{itemize}
We call such an action \emph{normalized}. 
The basis $\{a,b,c,d,e\}$ induces a basis of  $\bigwedge^{3} W$, 
with Pl\"{u}cker coordinates as basis elements. The {\em $\rho$-weight} of a quadratic monomial $m=xy$ is 
$w_{\rho}(m)=\bar{x}+\bar{y}$ 
and the $\rho$-weight of a Pl\"{u}cker coordinate $m_1\wedge m_2 \wedge m_3$ 
is $\sum_{i=1}^3w_{\rho}(m_i)$. We say that a net $\Lambda$ is \emph{$\rho$-semi-stable} 
(resp., \emph{$\rho$-stable}) if there exists a Pl\"{u}cker coordinate that does not vanish on $[\Lambda]$ 
with {\em non-negative} (resp., {\em positive}) $\rho$-weight. 
With this notation, the numerical criterion simply asserts that $\Lambda$ is semi-stable 
(resp., stable) if and only if $\Lambda$ is $\rho$-semi-stable (resp., $\rho$-stable) 
for all 1-PS's. 

\emph{A priori}, the numerical criterion requires one to check $\rho$-semi-stability for all 1-PS's. However, 
there necessarily exists a \emph{finite} set of numerical types of 1-PS's 
$\{\rho_i\}_{i=1}^{N}$ such that the union of the $\rho_i$-unstable points is 
$\GG(3,15) \setminus \GG(3,15)^{ss}$. The first main result of this section, Theorem \ref{T:subgroups}, 
describes such a set of 1-PS's explicitly. The second main result of this section, Theorem \ref{T:Description}, 
gives geometric characterizations of the nets destabilized by each $\rho_i$ in our list. Finally, in Section 
\ref{S:semi-stable}, we use this result to describe the semi-stable locus $\GG(3,15)^{ss} \subset \GG(3,15)$ 
explicitly.

\subsection{Notation and conventions}
Throughout this section, we use the following notation. 
Given a basis $\{a,b,c,d,e\}$ of $V$,  
we consider two orderings on the set of quadratic monomials in $W$. 
There is the lexicographic ordering, which is complete, and which we denote by $\succ_{lex}$. 
Then there is the ordering, denoted by $\geqslant$, according to which 
$m_1 \geqslant m_2$ if and only if $w_{\rho}(m_1)\geq w_{\rho}(m_2)$ for any normalized 1-PS acting diagonally on $\{a,b,c,d,e\}$. 
Note that 
\[
m_1\geqslant m_2 \Longrightarrow m_1\succeq_{lex} m_2.
\]

Finally, given a normalized 1-PS acting on $\{a,b,c,d,e\}$, there is a complete ordering 
$\succ_{\rho}$ on the quadratic monomials in $W$ defined as follows: 
$m_1\succ_{\rho} m_2$ if and only if one of the following conditions hold:
\begin{enumerate}
\item
$w_{\rho}(m_1)>w_{\rho}(m_2)$
\item 
$w_{\rho}(m_1)=w_{\rho}(m_2)$ and $m_1\succ_{lex} m_2$.
 \end{enumerate}
 
For any quadric $Q \in W$, we let $in_{lex}(Q)$ denote the initial monomial of $Q$ with respect to 
$\succ_{lex}$ and, if $\rho$ is a normalized 1-PS acting on $\{a,b,c,d,e\}$, 
we let $in_{\rho}(Q)$ denote the initial monomial of $Q$ with respect to $\succ_{\rho}$. 

For any net $\Lambda$, we may choose a basis 
$\{Q_1, Q_2, Q_3\}$ such that $in_{lex}(Q_1) \succ_{lex} in_{lex}(Q_2) \succ_{lex} in_{lex}(Q_3)$. 
We call such a basis of $\Lambda$ \emph{normalized}. 

Finally, given a basis $\{a,b,c,d,e\}$ of $V$, we define the {\em distinguished flag} 
$O \subset L \subset P \subset H \subset \PP V$ as follows:
\begin{align*}
O: &\ b=c=d=e=0, \\
L: &\ c=d=e=0, \\
P: &\ d=e=0, \\
H: &\ e=0.
\end{align*}

\subsection{Classification of destabilizing subgroups}
\label{S:subgroups}

\begin{theorem}\label{T:subgroups} 
 Suppose that $\Lambda$ is semi-stable with respect to every one-parameter subgroup
of the following numerical types:
\begin{enumerate}
\item $\rho_1=(1,1,1,1,-4)$. 
\item $\rho_2=(2,2,2,-3,-3)$.
\item $\rho_3=(3,3,-2,-2,-2)$.
\item $\rho_4=(4,-1,-1,-1,-1)$.
\item $\rho_{5}=(3,3,3,-2,-7)$.
\item $\rho_{6}=(4,4,-1,-1,-6)$.
\item $\rho_{7}=(9,4,-1,-6,-6)$.
\item $\rho_{8}=(7,2,2,-3,-8)$.
\item $\rho_{9}=(12,7,2,-8,-13)$.
\item $\rho_{10}=(9,4,-1,-1,-11)$.
\item $\rho_{11}=(14,4,-1,-6,-11)$.
\item $\rho_{12}=(13,8,3,-7,-17)$.
\end{enumerate}
Then $\Lambda$ is semi-stable. 
\end{theorem}
\begin{remark}\label{R:torus-remark} In fact, our proof gives a slightly stronger statement. Namely, 
$\Lambda$ is semi-stable with respect to a fixed torus $T$ if and only if it is 
semi-stable with respect to all one-parameter subgroups in $T$ of the numerical types $\{\rho_i\}_{i=1}^{12}$.
\end{remark}

\subsubsection*{Preliminary observations}

Fix a net $\Lambda$ which is $\rho_i$-semi-stable for each $\{\rho_i\}_{i=1}^{12}$. 
By the numerical criterion, to prove that $\Lambda$ is semi-stable, it suffices to show that $\Lambda$
is semi-stable with respect to an arbitrary 1-PS $\chi\co \CC^* \rightarrow \SL(5)$. 
Without loss of generality, we can assume that
$\chi$ is normalized, acting diagonally on the basis $\{a,b,c,d,e\}$ 
with weights $(\bar{a}, \bar{b}, \bar{c}, \bar{d}, \bar{e})$, 
satisfying $\bar{a}\geq \bar{b} \geq \bar{c} \geq \bar{d} \geq \bar{e}$.
To prove the theorem, we must exhibit a Pl\"{u}cker coordinate that does not vanish on $\Lambda$ and
has non-negative $\chi$\nb-weight. More explicitly, if $(Q_1, Q_2, Q_3)$ is a normalized basis of $\Lambda$, 
we must exhibit non-zero quadratic monomials $m_1, m_2, m_3$ in the 
variables $\{a,b,c,d,e\}$ which appear with non-zero coefficient in $Q_1\wedge Q_2 \wedge Q_3$ and satisfy 
$w_{\chi}(m_1)+w_{\chi}(m_2)+w_{\chi}(m_3) \geq 0$. 
We begin with two preparatory lemmas.

\begin{lemma}\label{obs}
The normalized basis $(Q_1, Q_2, Q_3)$ of $\Lambda$ satisfies the following:
\begin{enumerate}[(i)]
\item\label{O1}  $Q_3 \notin (e^2)$,
\item\label{O2}  $(Q_1, Q_2, Q_3) \not\subset (d,e)$, and either 
$(Q_2,Q_3)\notin(d,e)$ or $Q_3\notin (d,e)^2$.
\item\label{O3}  $(Q_2,Q_3) \not\subset (c,d,e)^2$, and either $(Q_1,Q_2,Q_3)\not\subset (c,d,e)$ 
or $Q_3\notin (c,d,e)^2$.
\item\label{O4}  $in_{lex}(Q_1)=a^2$ or $in_{lex}(Q_1), in_{lex}(Q_2) \in (a)$.
\end{enumerate}
\end{lemma}
\begin{proof}
(i), (ii), (iii), (iv) follow immediately from $\rho_1$, $\rho_2$, $\rho_3$, $\rho_4$-semi-stability of $\Lambda$, respectively.
\end{proof}

\begin{lemma}
\label{L:weight-b}
If $\bar{b} \leq 0$, then $\Lambda$ is $\chi$-semi-stable. 
\end{lemma}

\begin{proof}
First, suppose $in_{lex}(Q_1)=a^2$. Then $ in_{lex}(Q_2) \geqslant be$ and $in_{lex}(Q_3) \geqslant de$ 
by Lemma \ref{obs}\eqref{O3} and \eqref{O1}, respectively. 
In addition, $\rho_2$-semi-stability implies that if $Q_3 \in (d,e)^2$, then $Q_2 \notin (d,e)$. 
Thus, either $in_{lex}(Q_3) \geqslant ce$ or $Q_2$ contains a term $\geqslant c^2$. 
In the latter case, we obtain a Pl\"{u}cker coordinate of weight at least 
$2\bar{a}+2\bar{c}+\bar{d}+\bar{e}=-2\bar{b}-\bar{d}-\bar{e} > 0$, since $\bar{b} \leq 0$ and $\bar{e}<0$. 
In the former case, $\rho_1$-semi-stability implies that we cannot have $(Q_2, Q_3) \subset (e)$, 
so $Q_2$ or $Q_3$ contains a term $\geqslant d^2$. 
We obtain a Pl\"{u}cker coordinate of weight at least $2\bar{a}+2\bar{d}+\bar{c}+\bar{e}=-2\bar{b}-\bar{c}-\bar{e}>0$. 
Thus, $\Lambda$ is $\chi$-stable.

Next, suppose $in_{lex}(Q_1)\neq a^2$. 
Then we have $in_{lex}(Q_1) \geqslant ad$,  $in_{lex}(Q_2) \geqslant ae$, $in_{lex}(Q_3) \geqslant de$ 
by Lemma \ref{obs}\eqref{O1} and \eqref{O4}.  
If $in_{lex}(Q_2) \geqslant ad$, we are done since we have a Pl\"{u}cker coordinate of weight 
$2\bar{a}+\bar{c}+2\bar{d}+\bar{e}=-2\bar{b}-\bar{c}-\bar{e}>0$. 
Assume $in_{lex}(Q_2)=ae$. 
Then $\rho_7$-semi-stability implies that either $Q_3$ contains a term $\geqslant ce$ or $in_{lex}(Q_1) \geqslant ac$. 
In either case, we obtain a Pl\"{u}cker coordinate of weight at least 
$2\bar{a}+\bar{c}+\bar{d}+2\bar{e}=-2\bar{b}-\bar{c}-\bar{d} \geq 0$.
We conclude that 
$\Lambda$ is $\chi$-semi-stable.
\end{proof}

We can now begin the proof of the main theorem.
\begin{proof}[Proof of Theorem \ref{T:subgroups}]

We consider separately the following three cases:
\begin{enumerate}
\item[(I)] $O$ is not in the base locus of $\Lambda$;
\item[(II)] $O$ is in the base locus of $\Lambda$ but $L$ is not;
\item[(III)] $L$ is in the base locus of $\Lambda$.
\end{enumerate}

\subsection*{Case I: $O$ is not a base point.} 
We have $in_{lex}(Q_1)=a^2$. Lemma \ref{obs}\eqref{O3} implies that 
$in_{lex}(Q_2)\geqslant be$.
If $Q_3$ has a term $\geqslant cd$, then $\Lambda$ is $\chi$-stable because $2\bar{a}+(\bar{b}+\bar{e})+(\bar{c}+\bar{d})=\bar{a}>0$. 
We assume that $Q_3$ has no term $\geqslant cd$. By $\rho_5$-semi-stability, $Q_2$ has a term $\geqslant cd$. 
Now, if $Q_3$ has a term $\geqslant be$, then $\Lambda$ is again $\chi$-stable. 
So we assume that $Q_3$ has no term $\geqslant be$.
  
First, assume $Q_2$ has a term $\geqslant bd$. If $in_{lex}(Q_3) = ce$, then $\Lambda$ is $\chi$-stable 
since $2\bar{a}+(\bar{b}+\bar{d})+(\bar{c}+\bar{e})=\bar{a}>0$. Otherwise, $in_{lex}(Q_3)\in \{d^2, de\}$ and, 
by Lemma \ref{obs}\eqref{O2}, $Q_2$ must have a term $\geqslant c^2$.
Thus, if $\Lambda$ is not $\chi$-semi-stable, 
we must have $2\bar{a}+(\bar{b}+\bar{d})+(\bar{d}+\bar{e})=\bar{a}+\bar{d}-\bar{c}< 0$
and $2\bar{a}+2\bar{c}+(\bar{d}+\bar{e})=\bar{a}+\bar{c}-\bar{b}< 0$. 
This is clearly impossible.

From now on, we suppose $Q_2$ has no term $\geqslant bd$ and $Q_3\in (ce,d^2,de,e^2)$. 
Since $\Lambda$ is $\rho_6$-semi-stable,  
$Q_3$ contains a $d^2$ term. 
If, in addition, $in_{lex}(Q_3) = ce$, then recalling that
$Q_2$ has a term $\geqslant cd$, we have Pl\"{u}cker coordinates of weights at least
\begin{align*}
 2\bar{a}+(\bar{c}+\bar{d})+(\bar{c}+\bar{e})&=\bar{a}+\bar{c}-\bar{b},\\
2\bar{a}+(\bar{b}+\bar{e})+2\bar{d}&=\bar{a}+\bar{d}-\bar{c},\\
2\bar{a}+(\bar{b}+\bar{e})+(\bar{c}+\bar{e})&=\bar{a}+\bar{e}-\bar{d}.
\end{align*}
The three expressions cannot be simultaneously non-positive, so $\Lambda$ is $\chi$-stable.\\

It remains to consider the case $in_{lex}(Q_3)=d^2$. 
By Lemma \ref{obs}\eqref{O2}, $Q_2$ contains a term $\geqslant c^2$.
Thus, if $\Lambda$ is not $\chi$-semi-stable, 
we have $2\bar{a}+(\bar{b}+\bar{e})+2\bar{d}=\bar{a}+\bar{d}-\bar{c}< 0$ and 
$2\bar{a}+2\bar{c}+2\bar{d}=-2(\bar{b}+\bar{e})< 0$. This is clearly impossible.

\subsection*{Case II: $O$ is a base point but $L$ is not in the base locus}
\begin{claim}
Without loss of generality, we may assume $Q_1, Q_2, Q_3$ satisfy the following conditions:
\begin{enumerate}
\item $in_{lex}(Q_1), in_{lex}(Q_2) \in \{ab, ac, ad,ae\}$ and $Q_3\in (b,c,d,e)^2$.
\item $Q_1$ has a term $\geqslant b^2$, but $Q_2, Q_3$ have no term $ \geqslant b^2$.
\item $Q_3 \in (be, ce, d^2, de, e^2)$. 
\item $Q_2$ has a term $\geqslant cd$. 
\end{enumerate}
\end{claim}
\begin{proof}[Proof of Claim]
Indeed, (1) is immediate from Lemma \ref{obs}\eqref{O4} using the assumption that $O$ is a basepoint. 
If $Q_3\notin (b,c,d,e)^2$, then $in_{lex}(Q_3)\geqslant ae$ and 
$\Lambda$ is $\chi$-stable as $(\bar{a}+\bar{c})+(\bar{a}+\bar{d})+(\bar{a}+\bar{e}) >0$.

For (2), the assumption that $L$ is not in the base locus implies that $Q_1$, $Q_2$, or $Q_3$
 must have a term $\geqslant b^2$. 
We now deal with the case when $Q_2$ or $Q_3$ has a term $\geqslant b^2$. 
If $Q_3$ has a term $\geqslant b^2$, then $\Lambda$ is $\chi$-stable since 
$(\bar{a}+\bar{e})+(\bar{a}+\bar{d})+(2\bar{b})=\bar{a}+\bar{b}-\bar{c}\geq \bar{a} >0$. 
If $Q_2$ has a term $\geqslant b^2$, we consider two cases: If $Q_3$ has a term $\geqslant ce$, 
then we are done by Lemma \ref{L:weight-b} since $(\bar{a}+\bar{d})+2\bar{b}+(\bar{c}+\bar{e})=\bar{b}$. 
Otherwise, $Q_3\in (d,e)^2$. 
If $in_{lex}(Q_1)=ad$, then this contradicts $\rho_7$-semi-stability. 
Thus $in_{lex}(Q_1)\geqslant ac$. 
We are now done by Lemma \ref{L:weight-b} since $(\bar{a}+\bar{c})+(2\bar{b})+(\bar{d}+\bar{e})=\bar{b}$. 

Finally, to prove (3), we recall that $Q_3\in (b,c,d,e)^2$.
By Lemma \ref{L:weight-b}, $Q_3$ is $\chi$-semi-stable if it has a term $\geqslant cd$, 
as $(2\bar{b})+(\bar{a}+\bar{e})+(\bar{c}+\bar{d})=\bar{b}$. 
\end{proof}

We subdivide the further analysis into six cases according to the initial monomials of $in_{lex}(Q_1)$ and 
$in_{lex}(Q_2)$.

\subsubsection*{Case II.1: $in_{lex}(Q_1)=ad$ and $in_{lex}(Q_2)=ae$.} 
By (2) $Q_2$ has no term $\geqslant b^2$. Since $\Lambda$ is  
$\rho_{11}=(14, 4, -1, -6, -11)$-semi-stable, we see that 
$Q_3 \notin (ce, d^2, de, e^2)$. It follows by (3) that $in_{lex}(Q_3)=be$.

By $\rho_8=(7,2,2,-3,-8)$-semi-stability, $Q_2$ has a term $\geqslant c^2$. 
We now consider two subcases, according to whether $Q_3$ has a $d^2$ term.

If $Q_3$ has a $d^2$ term,
we have Pl\"{u}cker coordinates of $\chi$-weights
\begin{align*}
(\bar{a}+\bar{d})+(\bar{a}+\bar{e})+(\bar{b}+\bar{e})&=\bar{a}+\bar{e}-\bar{c},\\
(\bar{a}+\bar{d})+2\bar{c}+(\bar{b}+\bar{e})&=\bar{c},\\
2\bar{b}+2\bar{c}+2\bar{d}&=-2(\bar{a}+\bar{e}).
\end{align*}
Evidently, these expressions cannot be simultaneously negative, so $\Lambda$ is $\chi$-semi-stable. 

If $Q_3$ has no term $\geqslant d^2$, then
by $\rho_{10}=(9,4,-1,-1,-11)$-semi-stability, $Q_2$ has a term $\geqslant bd$. 
Thus, we have Pl\"{u}cker coordinates with $\chi$-weights at least
\begin{align*}
(\bar{a}+\bar{d})+(\bar{a}+\bar{e})+(\bar{b}+\bar{e})&=\bar{a}+\bar{e}-\bar{c},\\
(\bar{a}+\bar{d})+2\bar{c}+(\bar{b}+\bar{e})&=\bar{c},\\
(\bar{a}+\bar{d})+(\bar{b}+\bar{d})+(\bar{b}+\bar{e})&=\bar{b}+\bar{d}-\bar{c}.
\end{align*} 

Evidently, these expressions cannot be simultaneously negative so $\Lambda$ is $\chi$-semi-stable.

\subsubsection*{Case II.2: $in_{lex}(Q_1)=ac$ and $in_{lex}(Q_2)=ae$.}
Since $Q_2$ has no $b^2$ term, $\rho_7$-semi-stability implies $in_{lex}(Q_3)\geqslant ce$.

Suppose first $Q_2$ has a term $\geqslant bd$. By $\rho_{10}$-semi-stability, 
either $Q_3$ has a $be$ term or $Q_3$ has a $d^2$ term.
If $Q_3$ has a $be$ term, then $\Lambda$ is semi-stable by Lemma \ref{L:weight-b}
since $(\bar{a}+\bar{c})+(\bar{b}+\bar{d})+(\bar{b}+\bar{e})=\bar{b}$. 

If $Q_3$ has a $d^2$ term, then we have Pl\"{u}cker coordinates of weights at least
\begin{align*}
 (\bar{a}+\bar{c})+(\bar{a}+\bar{e})+(\bar{c}+\bar{e})&=-2(\bar{b}+\bar{d}),\\
  (\bar{a}+\bar{c})+(\bar{b}+\bar{d})+(\bar{c}+\bar{e})&=\bar{c},\\
  2\bar{b}+(\bar{a}+\bar{e})+2\bar{d}&=\bar{b}+\bar{d}-\bar{c}. 
 \end{align*}
 Evidently, these cannot all be negative so $\Lambda$ is $\chi$-semi-stable.  
 
 Suppose now $Q_2$ has no term $\geqslant bd$.
 Then by $\rho_{5}$-semi-stability $Q_2$ has a term $\geqslant cd$ and
 by $\rho_{10}$-semi-stability $Q_3$ has a $d^2$ term.
 
  Finally, by $\rho_{12}=(13,8,3,-7,-17)$-semi-stability, either $Q_2$ has a $c^2$ term
 or $Q_3$ has $be$ term.

 If $Q_3$ has $be$ term, recalling that $Q_1$ has a $b^2$ term, we have Pl\"{u}cker coordinates of weights at least
\begin{align*}
(\bar{a}+\bar{c})+(\bar{a}+\bar{e})+(\bar{b}+\bar{e})&=\bar{a}+\bar{e}-\bar{d}, \\
(\bar{a}+\bar{c})+(\bar{a}+\bar{e})+(2\bar{d})&=\bar{a}+\bar{d}-\bar{b}, \\
(\bar{a}+\bar{c})+(\bar{c}+\bar{d})+(\bar{b}+\bar{e})&=\bar{c}.\\
2\bar{b}+(\bar{a}+\bar{e})+(2\bar{d})&=\bar{b}+\bar{d}-\bar{c}.
\end{align*}
One of these is non-negative so $\Lambda$ is $\chi$-semi-stable. 

If $Q_2$ has a $c^2$ term, then we have Pl\"{u}cker coordinates of weights at least
   \begin{align*}
    (\bar{a}+\bar{c})+(\bar{a}+\bar{e})+(\bar{c}+\bar{e})&=-2(\bar{b}+\bar{d}),\\
  2\bar{b}+(\bar{a}+\bar{e})+2\bar{d}&=\bar{b}+\bar{d}-\bar{c}, \\
    2\bar{b}+2\bar{c}+2\bar{d}&=-2(\bar{a}+\bar{e}).
 \end{align*}
One of these is non-negative so $\Lambda$ is $\chi$-semi-stable.

\subsubsection*{Case II.3: $in_{lex}(Q_1)=ac$ and $in_{lex}(Q_2)=ad$.} 
By $\rho_7$-semi-stability, $in_{lex}(Q_3)\geqslant ce$. 
Hence, there is a Pl\"{u}cker coordinate of weight at least 
$2\bar b +(\bar a+\bar d)+(\bar c+\bar e)=\bar b$ and we are done by Lemma \ref{L:weight-b}.

\subsubsection*{Case II.4: $in_{lex}(Q_1)=ab$ and $in_{lex}(Q_2)=ae$.}
By $\rho_5$-semi-stability, $Q_2$ has a term $\geqslant cd$. 
If $in_{lex}(Q_3)=be$, then we are done by Lemma \ref{L:weight-b} 
as $(\bar{a}+\bar{b})+(\bar{c}+\bar{d})+(\bar{b}+\bar{e})=\bar{b}$.
Assume $Q_3 \in (ce,d^2,de,e^2)$. 
We consider the following three subcases cases:

Suppose $Q_3$ has $d^2$ term but no $ce$ term. Then $Q_2$ has a term 
$\geqslant c^2$ by Lemma \ref{obs}\eqref{O2}.
We have Pl\"{u}cker coordinates with weights at least
 $(\bar{a}+\bar{b})+(\bar{a}+\bar{e})+2\bar{d}=\bar{a}+\bar{d}-\bar{c}$ 
 and $(\bar{a}+\bar{b})+2\bar{c}+2\bar{d}=\bar{c}+\bar{d}-\bar{e}\geq \bar{c}$. 
 Evidently, these cannot both be negative.

Suppose $Q_3$ has both $d^2$ and $ce$ terms.
We have Pl\"{u}cker coordinates with weights 
 $(\bar{a}+\bar{b})+(\bar{a}+\bar{e})+2\bar{d}=\bar{a}+\bar{d}-\bar{c}$ 
 and $(\bar{a}+\bar{b})+(\bar{c}+\bar{d})+(\bar{c}+\bar{e})=\bar{c}$. Evidently, these 
 cannot both be negative.

Finally, suppose $Q_3$ has no $d^2$ term. Then 
by $\rho_{6}$-semi-stability $Q_2$ has a term $\geqslant bd$.
If $Q_3$ has a $ce$ term, then we are done by Lemma \ref{L:weight-b} 
as $(\bar{a}+\bar{b})+(\bar{b}+\bar{d})+(\bar{c}+\bar{e})=\bar{b}$. 
We may now assume $Q_3 =de$.  By Lemma \ref{obs}\eqref{O2}, $Q_2$ has a term $\geqslant c^2$.
Then we have Pl\"{u}cker coordinates with weights at least
\begin{align*}
(\bar{a}+\bar{b})+(\bar{a}+\bar{e})+(\bar{d}+\bar{e})&=\bar{a}+\bar{e}-\bar{c}, \\  
(\bar{a}+\bar{b})+2\bar{c}+(\bar{d}+\bar{e})&=\bar{c}, \\
(\bar{a}+\bar{b})+(\bar{b}+\bar{d})+(\bar{d}+\bar{e})&=\bar{b}+\bar{d}-\bar{c}.
\end{align*}
Evidently, these cannot be simultaneously negative. 

\subsubsection*{Case II.5: $in_{lex}(Q_1)=ab$ and $in_{lex}(Q_2)=ad$.}
If $Q_3$ has a term $\geqslant ce$, then $\Lambda$ is $\chi$-stable since 
$(\bar{a}+\bar{b})+(\bar{a}+\bar{d})+(\bar{c}+\bar{e})=\bar{a}>0$. 
Assume $Q_3\in (d,e)^2$. By Lemma \ref{obs}\eqref{O2}, $Q_2$ has a term $\geqslant c^2$. 
So we have Pl\"{u}cker coordinates of weights at least
$(\bar{a}+\bar{b})+2\bar{c}+(\bar{d}+\bar{e})=\bar{c}$ and 
$(\bar{a}+\bar{b})+(\bar{a}+\bar{d})+(\bar{d}+\bar{e})=\bar{a}+\bar{d}-\bar{c}$. 
Evidently, these cannot be simultaneously negative.

\subsubsection*{Case II.6:} Finally, if $in_{lex}(Q_1)=ab$ and $in_{lex}(Q_2)=ac$, then
$\Lambda$ is $\chi$-stable since $(\bar{a}+\bar{b})+(\bar{a}+\bar{c})+(\bar{d}+\bar{e})=\bar{a}>0$.

\subsection*{Case III: $L$ is in the base locus.}

\begin{claim}
Without loss of generality, we may assume $Q_1, Q_2, Q_3$ satisfy the following conditions:
\begin{enumerate}
\item $in_{lex}(Q_1), in_{lex}(Q_2) \in \{ac,ad,ae\}$.
\item $in_{lex}(Q_3) \in \{bd, be\}$.
\end{enumerate}
\end{claim}
\begin{proof}
Indeed,  (1) follows from Lemma \ref{obs}\eqref{O4} using $Q_1, Q_2 \in (c,d,e)$. For (2), note that 
Lemma \ref{obs}\eqref{O3} implies $Q_3 \notin (c,d,e)^2$, 
and that $\Lambda$ is $\chi$-stable if either $in_{lex}(Q_3)\geqslant bc$ or $in_{lex}(Q_3)\geqslant ae$.  
\end{proof}

\noindent
We consider three cases according to the initial monomials $in_{lex}(Q_1)$ and $in_{lex}(Q_2)$:

\smallskip 

{\em Case III.1:} Suppose $in_{lex}(Q_1)=ac$, $in_{lex}(Q_2)=ad$. Then $\Lambda$ is $\chi$-stable since
$(\bar{a}+\bar{c})+(\bar{a}+\bar{d})+(\bar{b}+\bar{e})
=\bar{a}>0.$

\smallskip 

{\em Case III.2:} 
Suppose $in_{lex}(Q_1)=ac$, $in_{lex}(Q_2)=ae$. If $in_{lex}(Q_3)=bd$, then $\Lambda$ is $\chi$-stable.
Assume $in_{lex}(Q_3)=be$. 
By $\rho_6$-semi-stability, $Q_2$ has a term $\geqslant bd$.
Since  
$(\bar{a}+\bar{c})+(\bar{b}+\bar{d})+(\bar{b}+\bar{e})=\bar{b}$, 
we are done by Lemma \ref{L:weight-b}.

\smallskip 

{\em Case III.3:}   $in_{lex}(Q_1)=ad$, $in_{lex}(Q_2)=ae$.  We consider separately two 
subcases: $in_{lex}(Q_3)=bd$ and $in_{lex}(Q_3)=be$.

{\em Case III.3(a):} If $in_{lex}(Q_3)=bd$, there is a Pl\"{u}cker coordinate of weight 
$(\bar{a}+\bar{d})+(\bar{a}+\bar{e})+(\bar{b}+\bar{d})=\bar{a}+\bar{d}-\bar{c}$.
Thus, $\Lambda$ will be semi-stable 
if we can find a Pl\"{u}cker coordinate of weight at least $\bar{c}$. 
By Lemma \ref{obs}\eqref{O2}, one of $Q_1,Q_2,Q_3$ contains a term $\geqslant c^2$. 
If $Q_3$ contains a term $\geqslant c^2$, 
then we have a Pl\"{u}cker coordinate of weight at least 
$(\bar{a}+\bar{d})+(\bar{a}+\bar{e})+2\bar{c}=\bar{a}-\bar{b}+\bar{c} \geq \bar{c}$. 
If $Q_2$ contains a term $\geqslant c^2$, then we have a Pl\"{u}cker coordinate of weight 
$(\bar{a}+\bar{d})+2\bar{c}+(\bar{b}+\bar{d})=\bar{c}+\bar{d}-\bar{e} \geq \bar{c}$. If $Q_1$ contains a term $\geqslant c^2$, 
then we have a Pl\"{u}cker coordinate of weight $(2\bar{c})+(\bar{a}+\bar{e})+(\bar{b}+\bar{d})=\bar{c}$. 
We are done.

{\em Case III.3(b):}  If $in_{lex}(Q_3)=be$, we use $\rho_6$-semi-stability to see that $Q_2$ has a term
$\geqslant bd$. Therefore, there is a Pl\"{u}cker coordinate of weight 
$(\bar{a}+\bar{d})+(\bar{a}+\bar{e})+(\bar{b}+\bar{e})=\bar{a}+\bar{e}-\bar{c}$ 
and a Pl\"{u}cker coordinate of weight at least 
$(\bar{a}+\bar{d})+(\bar{b}+\bar{d})+(\bar{b}+\bar{e})=\bar{b}+\bar{d}-\bar{c}$. 
To prove semi-stability, it suffices to exhibit a Pl\"ucker coordinate of weight at least $\bar{c}$.

By Lemma \ref{obs}\eqref{O2}, one of $Q_1,Q_2,Q_3$ contains a term $\geqslant c^2$. 
If $Q_3$ contains a term $\geqslant c^2$, 
then we have a Pl\"{u}cker coordinate of weight 
$(\bar{a}+\bar{d})+(\bar{a}+\bar{e})+2\bar{c}=\bar{a}+\bar{c}-\bar{b} \geq \bar{c}$. 
If $Q_2$ contains a term $\geqslant c^2$,  then we have a Pl\"{u}cker coordinate of weight 
$(\bar{a}+\bar{d})+(2\bar{c})+(\bar{b}+\bar{e})=\bar{c}$. 
It remains to consider the case when only $Q_1$ has a term $\geqslant c^2$. 
By $\rho_8$-semi-stability, $Q_3$ has a term $\geqslant cd$, and 
by $\rho_{9}$-semi-stability $Q_1$ has a term $\geqslant bc$. 
We obtain a Pl\"{u}cker coordinate with weight $(\bar{b}+\bar{c})+(\bar{a}+\bar{e})+(\bar{c}+\bar{d})=\bar{c}$.
At last, we are done.
\end{proof}

\subsection{Classification of Unstable Points}
\label{S:unstable-points}

In this section, we give geometric description of the strata in $\Gr(3,15)$ destabilized by each of the 1-PS's 
enumerated in Theorem \ref{T:subgroups}. For ease of exposition, we analyze the first four 1-PS's separately 
from the final eight.

\begin{lemma}\label{L:Weights1}
A net $\Lambda$ is destabilized by one of $\{\rho_i\}_{i=1}^{4}$ iff it satisfies one of the following conditions 
with respect to a distinguished flag $O \subset L \subset P \subset H \subset \PP^4$.
\begin{enumerate}
\item $\rho_1=(1,1,1,1,-4)$:
\begin{enumerate} 
\item A pencil of $\Lambda$ contains $H$, or
\item An element of the net is singular along $H$.
\end{enumerate} 
\item $\rho_2=(2,2,2,-3,-3)$: 
\begin{enumerate} 
\item $\Lambda$ contains $P$, or
\item A pencil of $\Lambda$ contains $P$, and an element of the pencil is singular along $P$.
\end{enumerate} 
\item $\rho_3=(3,3,-2,-2,-2)$:
\begin{enumerate} 
\item $\Lambda$ contains $L$,  and an element of $\Lambda$ is singular along $L$, or
\item A pencil of $\Lambda$ is singular along  $L$. 
\end{enumerate} 
\item $\rho_4=(4,-1,-1,-1,-1)$:
\begin{enumerate} 
\item $\Lambda$ contains $O$, and a pencil of $\Lambda$ is singular at $O$.\\
\end{enumerate} 
\end{enumerate}
\end{lemma}
\begin{proof}
In case (4), the only triple of initial $\rho_4$-weights with negative sum is $(3, -2,-2)$ and $(-2,-2,-2)$.
However, the stratum of nets with initial weights $(-2,-2,-2)$ 
is in the closure of the stratum of nets with initial weights 
$(3,-2,-2)$. Evidently, any quadric of weight $3$ contains $O$, 
while any quadric of weight $-2$ is singular at $O$. 
Thus, the net with initial $\rho_4$-weights $(3,-2,-2)$ 
has a base point at $O$ and contains a pencil of quadrics singular at 
$O$. The proofs of cases (1)--(3) are similar.
\end{proof}

On the basis of this partial analysis, we may already conclude the important fact that a semi-stable net has 
a pure one-dimensional intersection, and hence defines a connected curve with local complete intersection 
singularities.
\begin{corollary}\label{C:net=curve}
If a net of quadrics in $\PP^4$ is semi-stable, 
then the corresponding intersection is connected and purely one-dimensional.
\end{corollary}
\begin{proof} Fulton-Hansen connectedness theorem \cite{fulton-hansen} gives the first statement. 
If the intersection fails to be purely one-dimensional, then either a pencil of quadrics in the net contains a hyperplane,
in which case the net is destabilized by $\rho_1$, or we may 
choose a basis $\{Q_1, Q_2, Q_3\}$ of the net, 
such that $S:=Q_1 \cap Q_2$ is a quartic surface and one of its irreducible components 
is contained entirely in $Q_3$. Because $R_S:=\CC[a,b,c,d,e]/(Q_1,Q_2)$ 
has dimension $3$, $R_S$ is Cohen-Macaulay and so the ideal $(Q_1,Q_2)$ is saturated. 
Thus $S$ cannot lie entirely inside $Q_3$. We conclude that there must be an irreducible 
component $S' \subset S$ of degree at most $3$ which is contained in $Q_3$. 
If $\deg S'=1$, then the net contains a plane and is 
thus destabilized by $\rho_2$. If $\deg S'=2$, then the span of $S'$ is a hyperplane
and a pencil of quadrics in the net contains this hyperplane. 
Such a net is destabilized by $\rho_1$. 

Finally if $\deg S'=3$, 
then the classical classification of surfaces of minimal degree due to del Pezzo \cite{del-pezzo-1886}
implies that $S'$ is a \emph{rational normal scroll}; see \cite{eisenbud-harris-minimal-degree} for 
a modern proof of this result and \cite{harris} for an introduction to scrolls.
We have two cases to consider. If $S'$ is smooth, 
then the net is projectively equivalent to $(ad-bc, ae-bd, ce-d^2)$ (see \cite[Lecture 9]{harris}) 
and is destabilized by $\rho_3$. 
If $S'$ is singular, then it must be a cone over a rational normal cubic curve. 
If $O$ denotes the vertex of the cone, then we must have a pencil 
of quadrics singular at $O$. Such a net is destabilized by $\rho_4$.

\end{proof}

The following lemma is an unenlightening but straightforward combinatorial stepping stone to the geometric 
analysis in Theorem \ref{T:Description}.
\begin{lemma}\label{L:Weights2} 
Suppose $\Lambda$ is $\rho_k$-unstable for $k\in \{5,\dots, 12\}$
but is $\rho_j$-semi-stable for $1\leq j\leq k-1$. 
Let $m_1, m_2, m_3$ be the initial monomials of $\Lambda$ with respect to $\rho_k$.
Then $(w_{\rho_k}(m_1), w_{\rho_k}(m_2), w_{\rho_k}(m_3))$ must be one of the following triples:
\begin{enumerate}

\item[(5)] $\rho_5=(3,3,3,-2,-7)$:
\begin{itemize}
\item $(6,-4,-4)$
\end{itemize}

\item[(6)] $\rho_6=(4,4,-1,-1,-6)$: 
\begin{itemize}
\item$(8,-2,-7)$
\item $(3,-2,-2)$
\end{itemize}

\item[(7)] $\rho_{7}=(9,4,-1,-6,-6)$:
\begin{itemize}
\item $(8,3,-12)$
\end{itemize}

\item[(8)] $\rho_8=(7,2,2,-3,-8)$:
\begin{itemize}
\item $(4,-1,-6)$
\end{itemize}

\item[(9)] $\rho_{9}=(12,7,2,-8,-13)$:
\begin{itemize}
\item $(4,-1,-6)$
\end{itemize}

\item[(10)] $\rho_{10}=(9,4,-1,-1,-11)$:
\begin{itemize}
\item $(8,-2,-7)$
\end{itemize}

\item[(11)] $\rho_{11}=(14,4,-1,-6,-11)$:
\begin{itemize}
\item $(8,3,-12)$

\end{itemize}
\item[(12)] $\rho_{12}=(13,8,3,-7,-17)$:
\begin{itemize}
\item $(16, -4, -14)$
\end{itemize}

\end{enumerate}
\end{lemma}

\begin{proof}
The proof is purely algorithmic. Consider $\rho_k=(\bar a,\bar b,\bar c,\bar d,\bar e)$ for $5\leq k\leq 12$
and suppose $w_1 \geq w_2 \geq w_3$ is the triple of $\rho_k$-initial weights of 
a $\rho_k$-unstable net $\Lambda$. 
Lemma \ref{obs} translates into the following 
conditions satisfied by $w_1, w_2, w_3$:
\begin{enumerate}
\item[(C1)] If $\bar d \neq \bar e$, then $w_3>2\bar e$.
\item[(C2)] $w_1\geq 2c$. Moreover, if $w_2<2\bar c$, then $w_3\geq \bar c+\bar e$.
\item[(C3)] $w_2\geq \bar b+\bar e$. Moreover, if $w_1<2\bar b$, then $w_2\geq \bar a+\bar e$ and 
$w_3\geq \bar b+\bar e$. 
\item[(C4)] If $w_1\neq 2\bar a$, then $w_1\geq \bar a+\bar d$ and $w_2\geq \bar a +\bar e$.

\end{enumerate}
Now for each $\rho_k$, we list all triples of $\rho_k$-initial weights that have negative sum and
satisfy (C1)--(C4). We will do only Case (12), by far the most involved, 
and leave the rest as an exercise to the reader. 

The set of possible $\rho_{12}=(13,8,3,-7,-17)$-weights of quadratic monomials is 
\[
\{26,21,16,11,6,1,-4,-9,-14,-24, -34\}.
\] 
Suppose $w_1\geq w_2 \geq w_3$ are initial $\rho_{12}$-weights of $\rho_{12}$-unstable $\Lambda$ and
$\Lambda$ is $\rho_{i}$-stable for $i=1,\dots,4$. By (C1), $w_3\geq -24$.
If $w_1=26$, then the triples with negative sum are $(26,-14,-14)$, which violates
(C3), and $(26,-4,-24)$, which violates (C2). Suppose $w_1<26$. Then $w_1\geq 6$ and 
$w_2\geq -4$ by (C4). The triples with negative sum satisfying these conditions are 
\begin{itemize}
\item $(21,1,-24)$, which violates (C2);
\item $(16,-4,-14)$;
\item $(16,6,-24)$; 
\item $(11,-4,-9)$;
\item $(11,1,-14)$, which violates (C3);
\item $(6,-4,-4)$;
\item $(6,1,-9)$;
\item $(6,6,-14)$, which violates (C3).
\end{itemize}
Finally, one can easily check that the following statements hold: 
A net with $\rho_{12}$-initial weights $(16,6,-24)$ 
is $\rho_7=(9,4,-1,-6,-6)$-unstable. 
A net with $\rho_{12}$-initial weights $(11,-4,-9)$ 
is $\rho_8=(7,2,2,-3,-8)$-unstable.  
A net with $\rho_{12}$-initial weights $(6,-4,-4)$ 
is $\rho_{9}=(12,7,2,-8,-13)$-unstable. 
A net with $\rho_{12}$-initial weights $(6,1,-9)$ 
is destabilized by $\rho_{8}=(7,2,2,-3,-8)$ after the coordinate change $c\leftrightarrow b$.
\end{proof}

\begin{theorem}\label{T:Description}
A curve $C$ which is a complete intersection of three quadrics is unstable if and only if 
it is (a degeneration of) one of the following curves:
\begin{enumerate}
\item[(1)]\label{D1} $C$ is a double structure on an elliptic quartic curve in $\PP^3$.
\item[(2)] $C$ consists of a union of a double conic and two conics.
\item[(3)] $C$ has a non-reduced structure along a line $L$ 
and the residual curve $C'$ meets $L$ in at least $\deg C' -2$ points.
\item[(4)] $C$ has a point $O$ with a three dimensional Zariski tangent space, i.e. $C$ is not locally planar.
\item[(5)] $C$ contains a degenerate double structure on a conic, 
i.e. the double structure is contained in $\PP^3$.
\item[(6)] \begin{enumerate}
\item[]
\item $C$ consists of a union of an elliptic quartic curve and two conics meeting along a pair of triple points.
\item $C$ contains a double line that meets the residual arithmetic genus one component in three points.
\end{enumerate}
\item[(7)]  $C$ consists of two elliptic quartics meeting in an $A_5$ singularity and a node. 
\item[(8)] $C$ contains a planar $4$-fold point whose two branches are lines.
\item[(9)] $C$ contains a double line and the residual genus two curve is tangent to it.
\item[(10)]  $C$ consists of two tangent conics and an elliptic quartic meeting the conics in a $D_6$ singularity and two nodes.
\item[(11)]
$C$ has $D_5$ singularity and the hyperelliptic involution on the normalization of $C$ exchanges
the points lying over the $D_5$ singularity. 
\item[(12)] $C$ contains a conic meeting the residual genus two component in an $A_7$ singularity 
and the attaching point is a Weierstrass point on the genus two component.
\end{enumerate}
\end{theorem}

\begin{proof}
For each of $\{\rho_i\}_{i=1}^{12}$, we shall give a geometric description of the $\rho_i$-unstable stratum. 
Suppose $C$ is a $1$-dimensional complete intersection of three quadrics and let $\Lambda$
be its homogeneous ideal.
The analysis for $\{\rho_i\}_{i=1}^{4}$ proceeds via Lemma \ref{L:Weights1} (note that parts (1a) and (2a)
of the lemma do not apply as they describe intersections with a higher-dimensional component).

\smallskip

(1) By Lemma \ref{L:Weights1} (1b), a curve 
$C$ is $\rho_1$-unstable if and only if $\Lambda$ contains a double hyperplane. 
If $\Lambda$ contains a double hyperplane, then $C$ is a non-reduced curve with a double structure along an
elliptic quartic curve in $\PP^3$. 
Conversely, given a curve $C$ with a double structure along a necessarily degenerate elliptic quartic, 
let $H$ be the hyperplane containing $C_{red}$.
Since the restriction of $\Lambda$ to $H$ is at most two dimensional, we must have an element 
$Q\in \Lambda$
which contains $H$. 
If $\rank Q=2$, 
then our curve would be a reducible union of two degenerate quartic curves. We conclude that $\rank Q=1$,
and so $\Lambda$ contains a double hyperplane.

\smallskip

(2) By Lemma \ref{L:Weights1} (2b), $C$ is $\rho_2$-unstable if and only if a pencil of $\Lambda$ 
contains a plane $P$ and an element of the pencil is singular along $P$. 
Let $Q_2, Q_3$ generate the pencil, with $Q_3$ singular along $P$.
We have $Q_3=H_1\cup H_2$, a union of two hyperplanes, with $P=H_1 \cap H_2$. 
Since $P \subset Q_2$, we must have $Q_2 \cap H_1 = P \cup P_1$ 
and $Q_2 \cap H_2=P \cup P_2$ where $P_1$ and $P_2$ are planes. 
In sum, $Q_2 \cap Q_3=P_1 \cup P \cup P_2$, where 
the plane $P$ occurs in the intersection with multiplicity two. 
It follows that $C= Q_1 \cap Q_2 \cap Q_3$ consists of the union of two conics 
($Q_1 \cap P_1$ and $Q_1 \cap P_2$) 
and a double conic ($Q_1 \cap P$). 
Conversely, given such a curve, 
if we let $H_1$ and $H_2$ denote the hyperplanes spanned by each reduced conic with the double conic,
then $H_1 \cup H_2$ contains the curve, 
so we recover an element $Q_3$ of the net singular along $P$, the span of the double conic. 
Furthermore, since all elements of the net contain the double conic, the quadrics containing $P$ form a pencil.

\smallskip

(3)  By Lemma \ref{L:Weights1} (3), 
$C$ is $\rho_3$-unstable in two cases:
\begin{enumerate}
\item[(a)] The net contains a line $L$ and there is a quadric singular along $L$.
\item[(b)] There is a pencil of quadrics singular along a line $L$.
\end{enumerate}
Suppose there is a pencil of quadrics singular along a line $L$. Let $O'\in L\cap Q_3$. Then 
$O'$ is a base point of the net and a pencil of the net is singular at $O'$. It follows by Lemma 
\ref{L:Weights1} (4) that the net is destabilized by $\rho_4$.
Hence, it suffices to consider the case when 
the net contains $L$ and there is a quadric $Q_3$ singular along $L$. Then $C$ is generically non-reduced
along $L$. Let $C'$ be the subcurve of $C$ residual to $L$.
Note that every hyperplane containing $L$ intersects $Q_3$ in two planes and
the intersection of $Q_1\cap Q_2$ with each of these planes is the union of $L$ 
and at most a single other point. It follows
that every hyperplane containing $L$ intersects $C'$ in at least $\deg C' -2$ points lying on $L$. 
Therefore, $C'$ meets $L$ in $\deg C' -2$ points.

Conversely, suppose $C$ has a multiple structure along $L$ and the residual curve $C'$
intersects $L$ in $\deg C' -2$ points. Then the
projection of $C'$ away from $L$ is a conic. It follows that $C' \setminus L$ lies on a rank $3$ quadric $Q$
singular along $L$. Finally, a non-reduced component 
supported on $L$ lies on $Q$. Since $C$ is Cohen-Macaulay, it 
follows that $C$ lies on $Q$.

\smallskip

(4) By Lemma \ref{L:Weights1} (4), 
$C$ is $\rho_4$-unstable if and only if $O$ is a base point of the net, and the net contains a 
pencil of quadrics singular at $O$. 
If we choose generators $Q_1, Q_2, Q_3$ with $Q_2, Q_3$ singular at $O$, then  
$T_O C=T_O Q_1 \cap T_O Q_2 \cap T_O Q_3=T_O Q_1$ implies $\dim T_O C \geq 3$. Conversely, 
if $O\in C$ such that $\dim T_O C \geq 3$, 
then $O$ is a base point of the net and there is a pencil of quadrics singular at $O$.

\smallskip

Consider now $\{\rho_i\}_{i=5}^{12}$. 
For each triple of $\rho_i$-initial weights 
from Lemma \ref{L:Weights2}, the locus 
of nets having these initial weights is an irreducible locally closed set. In what follows we describe
the generic point of each of them.

(5)  $\rho_5=(3,3,3,-2,-7)$. By Lemma \ref{L:Weights2}, it suffices to consider 
a $\rho_5$-unstable net with initial $\rho_5$-weights $(6,-4,-4)$. Such a net has generators
$(Q_1,Q_2,Q_3)$ such that $Q_2=d^2+eL$, where $L$ is a linear form, and $Q_3\in (e)$. 
Evidently, $C$ contains the double conic $(Q_1, e, d^2)$ contained in the hyperplane $(e)$. 

Conversely, if $C$ contains a double structure on a conic contained in a hyperplane $H$, 
then we may take $(e)$ to be the ideal of $H$ and $(d,e)$ to be the ideal of the plane spanned by the 
underlying conic. It is clear that the restriction of $\Lambda$ to $H$ must contain the double plane $d^2$.
Thus the ideal of $C$ contains a quadric in $(e)$ and a linearly independent quadric in $(d^2)+(e)$.
It follows that $C$ is destabilized by $\rho_5$.

\smallskip

(6) $\rho_6=(4,4,-1,-1,-6)$. 
There are two possible triples of initial $\rho_6$-weights: $(8, -2, -7)$ and $(3,-2,-2)$.

If the initial weights are $(8,-2,-7)$, then there is a basis of $\Lambda$ 
of the form $(Q_1, Q_2, Q_3)$, where $Q_2\in (c,d)^2+(e)$ and $Q_3= eL(c,d,e)$.
If we let $H'$ be the hyperplane $L(c,d,e)=0$, then $H' \cap Q_1 \cap Q_2$ is an elliptic normal curve, 
while $H \cap Q_1 \cap Q_2$ is a pair of conics. All
three components meet in the two points of $L\cap Q_1$.  

Conversely, given a curve $C$ of this form, let $H'$ 
be the hyperplane spanned by the elliptic normal curve, 
and $H$ the hyperplane spanned by the pair of conics. Let $d=0$ and $e=0$ be the equations
of $H'$ and $H$, respectively.
Then $Q_3:=de \in \Lambda$, and the restriction of $\Lambda$ to $H$ contains a rank $2$ quadric 
singular along a line contained in $H'$. Thus we can choose the coordinate $c$ so that the ideal of $C$
can be written as $(Q_1, Q_2, Q_3)$, where $Q_2\in (c,d)^2+(e)$. 
Thus $C$ is destabilized by $\rho_6$.

If the initial weights are $(3,-2,-2)$, then
the net is generated by quadrics $Q_1\in (c,d,e)$ and  $Q_2, Q_3\in (c,d)^2+(e)$.
For a general such net, we can choose coordinates so that $(Q_2,Q_3)=(ae+c^2, be+d^2)$. This pencil 
cuts out a Veronese quadric with a double line along $c=d=e=0$ (cf. Lemma \ref{L:veronese}).
Being a quadric section of this Veronese, 
$C$ must be a union of a double line and an elliptic sextic meeting the double line in three points.

Conversely, suppose $C$ has a double line 
component meeting the residual component of arithmetic genus one in three points.
Take a quadric in the net with a vertex on the double line and let $e=0$ be the tangent hyperplane 
to this quadric. Then the scheme-theoretic intersection of $C$ with $e=0$ is a double line in $\PP^3$.
Assuming that the line is $c=d=e=0$, we conclude that in appropriately chosen coordinates 
the net is $(Q_1, Q_2, Q_3)$, where $Q_1\in (c,d,e)$, $Q_2=ae+c^2$, and $Q_3=be+d^2$. 
Such a net is destabilized by $\rho_6$. 

\smallskip 

(7) $\rho_{7}=(9,4,-1,-6,-6)$.
By Lemma \ref{L:Weights2}, we have only need to consider initial weights $(8, 3,-12)$.
Then the generators of the net can be written as 
\begin{align*}
Q_1&=ac+b^2+c^2 \mod{(d,e)}, \\
Q_2&=ad+bc \mod{(c,d,e)^2}, \\ 
Q_3&=de.
\end{align*}
The two elliptic quartics are $Q_1=Q_2=d=0$ and $Q_1=Q_2=e=0$.
Dehomogenizing with respect to $a$, we see that locally at $O$, 
we have $c=b^2+R_1$, $d=bc+R_2=b^3+bR_1+R_2$, 
and $de=0$. This translates into $(e-nb^3)e=0$ locally at $O$, which 
is an $A_5$ singularity. Restricting to $d=e=0$, we see that the two elliptic quartics 
intersect at $O$ and one other point. The claim follows.

\smallskip 

(8) $\rho_8=(7,2,2,-3,-8)$. A general net with $\rho_8$-initial weights $(4, -1, -6)$ 
has generators $Q_1=ad+\overline{Q}_1(b,c,d,e)$, $Q_2=ae+dL_1(b,c,d)+eL_2(b,c,d,e)$, and $Q_3=eL_3(b,c,d,e)+d^2$.  
After an appropriate change of variables, the generators can be rewritten as 
\begin{align*}
Q_1 &=ad+\overline{Q}_1(b,c,d,e),\\
Q_2 &=ae+bd, \\
Q_3 &=ce+d^2.
\end{align*}
Note that $\{d=e=0\}\cap C=\{\overline{Q}_1(b,c)=0\}$ is the union of two lines meeting at $O$. 
From the above, we deduce that $d=R(b,c)$ and $e=bd=bR(b,c)$ for some 
power series $R(b,c)$ with a quadratic initial form. Now $ce+d^2=0$ translates into
\[
R(b,c)(bc+R(b,c))=0,
\]
which defines an ordinary $4$-fold planar point.

Conversely, suppose $\Lambda$ is a net
defining a curve $C$ with a $4$-fold planar point $O$ whose two branches are lines. 
Let $L_1, L_2$ be the lines and $C'$ be the residual sextic. Then $C'$ has geometric genus one and
a node at $O$. Let $\pi\co \PP^4 \dashrightarrow \PP^3$ be the projection from $O$,
Set $C''=\pi(C')$ and $p_i=\pi(L_i)$ for $i=1,2$. Let $p_3$ and $p_4$ the images of the tangent lines
to the branches of $C'$ at $O$. Then $C''$ is a genus one quartic in $\PP^3$ and $p_1$, $p_2$ are points
lying on its chord $\overline{p_3p_4}$. Being a genus one quartic, $C''$ lies on a pencil of quadrics in $\PP^3$
and hence there is a quadric in $\PP^3$ containing $C''$ together with the $4$ collinear points 
$p_1,p_2,p_3,p_4$. This gives a rise to a singular quadric $Q_3$ in $\PP^4$ that
has a vertex at $O$, contains $C''$, and contains the plane $P$ spanned by $L_1$ and $L_2$. 
Since the quadrics in $\Lambda$ containing $P$ form a pencil, 
we conclude that there is a quadric $Q_2\in \Lambda$ that contains $P$ and which is 
linearly independent with $Q_3$.

Summarizing, we can choose coordinates so that $P$ is given by $d=e=0$ and find a basis 
of $\Lambda$ consisting of $Q_1=ad+\overline{Q}_1(b,c,d,e), Q_2=ae+bd, Q_3=dL+eM$, where
$L$ and $M$ are linear forms in $(b,c,d,e)$. 
The resulting local analytic equation at $O$ is 
$R(b,c)(bM(b,c,d,e)-L(b,c,d,e))=0$, where $R(b,c)$ is power series with a quadratic initial form. 
This equation defines a triple point unless $L(b,c,d,e)\in (d,e)$. Therefore $L(b,c,d,e)\in (d,e)$, and
$\Lambda$ is destabilized by $\rho_8=(7,2,2,-3,-8)$.

\smallskip

(9) $\rho_{9}=(12,7,2,-8,-13)$. By Lemma \ref{L:Weights2}, we have to consider nets with initial weights $(4,-1,-6)$. 
Such a net is generated by 
\begin{align*}
Q_1 &=ad+c^2 \mod (c,d,e)^2, \\
Q_2 &=ae+bd \mod (d,e)(c,d,e), \\
Q_3 &=be+cd \mod (d,e)^2.
\end{align*}
Restricting to $P$, we see that the net has a double structure along $L$ and $L$ meets the residual 
genus $2$ curve $D$ in a single point $O: b=c=d=e=0$. 
Dehomogenizing with respect to $a$, we see that the singularity at $O$ is locally analytically
$c^2(b^2-c)=0$. Thus $L$ is tangent to $D$ at $O$. 

Conversely, if a locally planar complete intersection $C$ contains a double line tangent to the residual
genus $2$ component, then after an appropriate change of coordinates, the ideal of $C$ is $(Q_1, Q_2, Q_3)$, 
where $Q_1=ad+c^2 \mod{(c,d,e)^2}$, $Q_2=ae+bd$, and $Q_3=be+cd$. Such a net is destabilized
by $(12, 7, 2, -8, -13)$.

\smallskip 

(10) $\rho_{10}=(9,4,-1,-1,-11)$.
Consider nets with initial $\rho_{10}$-weights $(8,-2,-7)$. The generators for such a net 
can be chosen to be 
$Q_1=ac+b^2+Q_1(b,c,d,e)$, $Q_2=ae+Q_2(c,d,e)$, $Q_4=be$.  The net defines a reducible 
curve. Along the plane $b=e=0$, the two components meet in three points defined by 
$Q_2(c,d)=ac+Q_1(c,d)=0$. 
Restricting to $e=0$, we obtain a reducible quartic $Q_2(c,d)=Q_1=0$, which is a union of two conics. 
Restricting to $b=0$, we obtain an elliptic quartic meeting the two conics in a $D_6$ singularity and two nodes.

\smallskip 

(11) $\rho_{11}=(14, 4, -1, -6, -11)$. By Lemma \ref{L:Weights2}, the only relevant triple of initial weights is
$(8, 3, -12)$.  
The general net with $\rho_{11}$-initial weights $(8,3,-12)$ is generated, after an appropriate change
of coordinates, by 
\begin{align*}
Q_1&=ad-b^2-R_1(c,d,e) \\
Q_2&=ae-bc-R_2(c,d,e) \\
Q_3&=ce-d^2.
\end{align*}
Dehomogenizing with respect to $a$, we can write the first two equations as
$d=b^2+R_1(c,d,e)$ and $e=bc+R_2(c,d,e)$. 

Now we plug into the equation $Q_3$ to get a local equation for the plane curve singularity at $O$: 
\[
bc^2+cR_2(c,d,e)-b^4-2b^2R_1(c,d,e)-R_1^2(c,d,e)=0,
\] 
which defines a $D_5$ singularity. 

Furthermore, the hyperelliptic involution induced on the normalization of $C$ by the projection away 
from $L$ interchanges the points lying over the singularity.

\smallskip 
(12) $\rho_{12}=(13, 8, 3, -7, -17)$. 
By Lemma \ref{L:Weights2}, the only possible triple of $\rho_{12}$-initial weights is
$(16, -4, -14)$. 
After an appropriate coordinate change, the net is generated by 
\begin{align*}
Q_1 &=ac+b^2+aR_1(d,e)+R_2(c,d,e), \\
Q_2 &= ae+cd, \\
Q_3 &=ce+d^2,
\end{align*}
where $R_1$ is a linear form and $R_2$ is a quadratic form.
The resulting curve has a conic component $C_1$ in the plane $P$ and $C_1$ meets the residual
component $C_2$ at the point $O$ in a singularity with the local analytic equation $d(b^4-d)=0$,
that is an $A_7$ singularity. By setting $c=-t^3, d=t^2, e=t$, we see that $C_2$
is given by the equation 
\[
b^2-t^3+R_1(t^2,t)+R_2(t^3, t^2, t)=0.
\]
In other words, the projection away from $L$ realizes $C_2$ as the genus two double cover of 
$\PP^1$ ramified at $O$. 
\end{proof}

\section{Geometry of semi-stable curves}
\label{S:semi-stable}

\subsubsection*{Notation}
To a net of quadrics in $\PP^4$ and a choice of its basis $(Q_1, Q_2, Q_3)$, we
associate the quintic polynomial $\det (xQ_1+yQ_2+zQ_3)$.
The $\PGL(3)$-orbit of the corresponding quintic plane curve
is an invariant of the net, which we call the \emph{discriminant quintic}.

Since a semi-stable net defines a complete intersection by Corollary \ref{C:net=curve}, 
we will use words ``net'' and ``curve'' interchangeably.
In particular, the discriminant $\Delta(C)$ of a semi-stable curve $C$ is the discriminant quintic of 
its defining net of quadrics.

\subsection{Main Results}
In this section, we use  the instability results of the previous section to give 
an explicit description of semi-stable curves.
Our main results are the following two theorems 
classifying reduced and non-reduced semi-stable curves. 

\begin{theorem}
\label{T:reduced-semi-stable}
A reduced semi-stable curve is a quadric section 
of a smooth quartic del Pezzo in $\PP^4$. 
Conversely, a quadric section $C$ of a smooth quartic del Pezzo in $\PP^4$ 
is unstable if and only if 
\begin{enumerate}
\item $C$ is non-reduced, or
\item $C$ is a union of an elliptic quartic and two conics meeting in a pair of triple points, or
\item $C$ is a union of two elliptic quartics meeting along an $A_5$ and an $A_1$ singularities, or
\item $C$ has a $4$-fold point with two lines as its two branches, or
\item $C$ is a union of two tangent conics and an elliptic quartic 
meeting the conics in a $D_6$ singularity and two nodes, or
\item $C$ has a $D_5$ singularity with pointed normalization $(\widetilde{C}, p_1, p_2)$ and
$p_1$ is conjugate to $p_2$ under the hyperelliptic involution of $\widetilde{C}$, or 
\item $C$ contains a conic meeting the residual genus $2$ component in an $A_7$ singularity 
and the attaching point of the genus $2$ component is a Weierstrass point, or
\item $C$ is a degeneration of curves in (1)--(7).
\end{enumerate} 
\end{theorem}

\begin{theorem}[Non-reduced semi-stable] 
\label{T:non-reduced-semi-stable}
Let $N\subset \MX$ be the image of the locus of non-reduced semi-stable curves. 
Then $N$ has the following decomposition into irreducible components:
\begin{equation*}
N=N_1 \cup N_2 \cup N_3 \cup N_4,
\end{equation*}
where
\begin{enumerate}
\item $N_1$ consists of a single point parameterizing
the balanced genus $5$ ribbon described by Equation \eqref{E:ribbon}.
\item $N_2$ parameterizes curves with a double twisted cubic meeting the residual conic in two points described by Equation \eqref{E:dtc}.
\item $N_3$ parameterizes curves with a double conic component 
meeting the residual rational normal quartic in three points described by Equation \eqref{E:double-conic}.
\item $N_4$ consists of a single point parameterizing
the semi-stable curve with two double lines joined by conics described by Equation \eqref{E:double-line}.
\end{enumerate}
\end{theorem}

Our analysis proceeds by investigation of the discriminant quintic 
and is motivated by the following easy result on the relationship between 
a curve and its discriminant:

\begin{lemma}\label{L:equivalence} Let $C$ be a complete intersection 
of three quadrics in $\PP^4$. Then $\Delta(C)$ is reduced if and only if 
$C$ lies on a smooth quartic del Pezzo in $\PP^4$.
\end{lemma} 

\begin{proof} Let $\Lambda$ be the net of quadrics containing $C$.
If $\Delta(C)$ is reduced, then $C$ lies on a smooth quartic del Pezzo, 
defined by any pencil $\ell \subset \Lambda$ transverse to $\Delta(C)$.
Conversely, if $C$ lies on a smooth quartic del Pezzo $P$ in $\PP^4$, then 
the pencil of quadrics containing $P$ has a reduced discriminant. 
It follows that $\Delta(C)$ is reduced.
\end{proof}

In Corollary \ref{C:smooth-DP} and Proposition \ref{P:reduced-discriminant-1}, 
we will show that 
a semi-stable curve $C$ is reduced if and only if its discriminant $\Delta(C)$ is reduced.  
This, together with Theorem \ref{T:Description},
leads to a fairly concrete description of reduced semi-stable curves 
as divisors on smooth del Pezzos given in Theorem \ref{T:reduced-semi-stable}.
On the other hand, if $C$ is non-reduced our analysis breaks into two cases, 
according to whether $\Delta(C)$ has a double line or a double conic. 
In each case, we find a distinguished quartic surface containing $C$, 
which enables us to describe $C$ rather explicitly. 
The surfaces arising in this analysis are described in Section \ref{S:quartics}.

\subsection{Special quartic surfaces in $\PP^4$}
\label{S:quartics}

Four quartic surfaces, each a complete intersection of two quadrics in $\PP^4$, 
play a special role in our analysis of semi-stable curves. Before describing them, 
let us briefly recall the classification of 
pencils of quadrics, or, equivalently quartic del Pezzo surfaces, by
their Segre symbols \cite{hodge-pedoe,miranda}:
\begin{definition} Let $\ell=\{Q(t) \mid t\in \PP^1\}$ be a pencil of quadrics in $\PP^4$, not all singular. 
Suppose $\ell$ has exactly $k$ singular elements $Q_1, \dots, Q_k$.
The {\em Segre symbol} of $\ell$ is a double array
$$\Sigma=\bigl((a_{ij})_{1\leq j\leq m_i}\bigr)_{1\leq i\leq k},$$ where 
$\sum_{j\geq r} a_{ij}$ is the minimum order of vanishing at $[Q_i]$ of $(6-r)\times (6-r)$ minors of $\ell$, 
considered as a function of $t$;
in particular, $\sum_{j=1}^{m_i}a_{ij}$
is the multiplicity of $[Q_i]$ in the discriminant $\Delta(\ell)$.
\end{definition}
Two quartic del Pezzo surfaces with projectively equivalent discriminants 
are projectively equivalent if and only if their Segre 
symbols are equal; see \cite[Theorem 2]{miranda} or \cite[p.278]{hodge-pedoe}. 
Therefore, if $\Sigma$ is a Segre symbol, we can speak of a del Pezzo surface 
$P(\Sigma)$.

\subsubsection{Special del Pezzos}
\label{S:del-pezzos}
We consider two special del Pezzo surfaces $P_0:=P(1,(1,1), (1,1))$ and $P_1:=P(1,2,2)$. 
We recall from \cite[Lemma 3]{miranda} that
$P_1$ is the 
anti-canonical embedding of the blow-up 
of $\PP^2$ at points $\{p, q_1, r_1, q_2, r_2\}$ on a smooth conic, with $r_i$  
infinitesimally close to $q_i$, for $i=1,2$; and that
$P_0$ is the anti-canonical
embedding of the blow up of $\PP^2$ at points $\{p, q_1, r_1, q_2, r_2\}$, where $r_i$ is
infinitesimally close to $q_i$, for $i=1,2$, and $p$ is the intersection of the lines $\overline{q_ir_i}$.

Note that $P_1$ isotrivially specializes to $P_0$. Indeed, if we choose coordinates $x,y,z$ on $\PP^2$ 
so that $z=0$ is the line $\overline{q_1q_2}$ and $x=0$ (resp., 
$y=0$) is the line $\overline{q_1r_1}$ (resp., $\overline{q_2r_2}$), then the degeneration can be realized by 
the one-parameter subgroup of $\operatorname{PGL}(3)$ acting on $\PP^2$ via $t\cdot [x:y:z]=[tx:ty:t^{-2}z]$.

\subsubsection{Veronese quartic} 
\label{S:veronese}
The third quartic surface of interest is described in the following lemma.
\begin{lemma}[Non-linearly normal Veronese] \label{L:veronese}
Let $V \subset \PP^4$ be the surface defined by the ideal $(ac-b^2, ce-d^2)$. 
Then $V$ is a projection of a Veronese surface in $\PP^5$. Moreover, $V$ 
has two pinch point singularities (local equation $uv^2=w^2$) at $[0:0:0:0:1]$ and $[1:0:0:0:0]$ 
as well as simple normal crossing along the line $b=c=d=0$. 
If an irreducible double quartic curve $C$ on $V$ is cut out by a quadric, 
then $C$ is either a double hyperplane section or $C$ is projectively equivalent to the balanced
ribbon $I_R=(ac-b^2, ce-d^2, ae-2bd+c^2)$.
\end{lemma}
\begin{proof}
Evidently, $V$ is the image of $[x:y:z] \mapsto [x^2:xy:y^2:yz:z^2]$,
which is a projection of the Veronese in $\PP^5$.
The statement about singularities follows from a local computation.

Every irreducible double quartic curve on $V$ must be the image of a double conic on $\PP^2$. 
Suppose that the conic has equation $f(x,y,z)=0$. The double quartic is a quadric section 
if and only if $f^2(x,y,z)\in \Sym^2 \CC[x^2, xy, y^2, yz, z^2]$. In particular, $f^2(x,y,z)$
cannot have $x^3z$ and $xz^3$ monomials. Thus either $f(x,y,z)$ has no $xz$ term or it has 
no $x^2$ and $z^2$ terms. In the former case, $f(x,y,z)=0$ is a hyperplane section of $V$. 
Suppose now $f(x,y,z)$ has no $x^2$ or $z^2$ term but has $xz$ term. Then
we can write $f(x,y,z)=xz +y L(x,z)+\lambda y^2$. 
After a linear change of variables on $\PP^2$ inducing a compatible linear change of variables in $\PP^4$, 
we can assume that $f(x,y,z)= xz+\lambda y^2$. If $\lambda=0$, then $f^2(x,y,z)=ae$ and it defines 
a union of two double conics. This contradicts the irreducibility assumption. 
Thus, we can assume that $f(x,y,z)=xz-y^2$, so that 
$(xz-y^2)^2=x^2z^2-2xzy^2+y^4=ae-2bd+c^2$.
\end{proof}

\subsubsection{Reducible quartic}
The final quartic of special interest to us is the reducible union of a plane with a cubic scroll, 
which arises in Part (1) of the following lemma.
\begin{lemma}
\label{L:reducible-quartic}
Suppose $\ell$ is a pencil of quadrics containing a common plane and with no common singular points.
Then $\ell$ is one of the following up to projectivity:
\begin{enumerate}
\item $\ell=(ad-bc, be-cd)$, defining a union of a plane with a cubic scroll.
The vertices of the quadrics in $\ell$ trace out the conic $b=d=ae-c^2=0$.
\item$\ell=(ad-\mu b^2, be-cd)$, where $\mu\in \CC$. 
The vertices of the rank $4$ quadrics in $\ell$ trace out the line $b=d=e=0$.
\end{enumerate}
\end{lemma}
\begin{proof}
Suppose a pencil of quadrics contains a plane $b=d=0$. Then the general form of the pencil is 
$bL_1-dL_2=bL_3-dL_4=0$.
Since the pencil does not have a common singular point, the set given by $b=d=L_1=L_2=L_3=L_4=0$ is 
empty. Without loss of generality, we can assume that $b,d,L_2,L_3,L_4$ are linearly independent and
choose coordinates so that $L_2=a$, $L_3=e$, $L_4=c$. Thus $\ell=(ad-bL_1, be-cd)$.
Changing coordinates, we can assume that $L_1=\lambda c+\mu b$. 
If $\lambda\neq 0$, a further change of coordinates:
$a':=\lambda a$, $c':=\lambda c+\mu b$, $b':=\lambda b$, $e':=e+(\mu/\lambda)d$, $d':=d$ gives
$I=(a'd'-b'c', b'e'-c'd')$. 

Finally, if $\lambda=0$, then the pencil is $(ad-\mu b^2, be-cd)$.
\end{proof}

\subsection{Non-reduced discriminants of semi-stable curves}
We proceed to give a complete classification of semi-stable curves with non-reduced discriminants.
An important implication of our analysis is the fact that a semi-stable curve with a non-reduced
discriminant is itself non-reduced. 

\subsubsection{Discriminants of pencils and nets of quadrics in $\PP^4$} 

We begin with a series of simple lemmas, whose proofs we omit.
\begin{lemma}\label{L:rank-4-singular}
Suppose $Q_1$ is a rank $4$ quadric. Then $\det(Q_1+tQ_2)$ has a root of multiplicity two at $t=0$,
or is identically zero, if and only if $Q_2$ vanishes at the vertex of $Q_1$.
\end{lemma}

\begin{lemma}\label{L:singular-pencil}
A pencil $\ell$ of quadrics in $\PP^4$ consists of singular quadrics only if:
\begin{enumerate}
\item[(A)] Quadrics in $\ell$ have a common singular point; or
\item[(B)] Quadrics in $\ell$ contain a common plane; or
\item[(C)] Restricted to a common hyperplane, the quadrics in $\ell$ are singular along a line; 
\end{enumerate}
Furthermore, if $\ell$ satisfies (C) but not (A) or (B), then up to projectivity $\ell=(ac-b^2, ce-d^2)$,
defining the Veronese quartic $V$.
\end{lemma}

\begin{lemma}\label{L:singular-quadrics-P3}
A pencil $\ell$ of quadrics in $\PP^3$ consists of singular quadrics if and only if:
\begin{enumerate}
\item[(A)] Quadrics in $\ell$ have a common singular point; or
\item[(B)] Restricted to a common plane, quadrics in $\ell$ contain a double line. The general 
such $\ell$ is, up to projectivity, $(be, ce-d^2)$.
\end{enumerate}
\end{lemma}

Next, we analyze the possibilities for non-reduced discriminants of semi-stable curves. 

\begin{prop}\label{P:discriminant-double-line}
A discriminant quintic of a semi-stable net $\Lambda$ has a double line 
if and only if (up to projectivity) $(ae-bd, ad-bc) \subset \Lambda$ and
$\Lambda$ is of the form 
\begin{equation*}
(ad-bc, be-cd, ae-c^2+bL_1+dL_2)
\end{equation*}
In particular, such $\Lambda$ contains a double conic.
\end{prop}

\begin{proof}
Suppose $\ell$ is a double line in $\Delta(C)$.
The analysis proceeds according to possibilities for $\ell$ enumerated
in Lemma \ref{L:singular-pencil}.

\smallskip

(A) Suppose all elements of $\ell$ are singular at a point $O$. 
By Lemma \ref{L:rank-4-singular}, $\ell$ can be a double line of $\Delta(C)$ in two cases: 
either $O$ is a base point of $\Lambda$ or 
all quadrics in $\ell$ have rank $\leq 3$. 
The former case is impossible by Lemma \ref{L:Weights1} (4). In the latter case
 Lemma \ref{L:singular-quadrics-P3} says that either all quadrics in $\ell$ are singular along a line, 
in which case $\Lambda$ is destabilized by Lemma \ref{L:Weights1} (3)
or $\ell$ is up to projectivity (a degeneration of) $(be, ce-d^2)$, in which case $\Lambda$ is
 destabilized by $\rho_5=(3,3,3,-2,-7)$.

\smallskip

(B) Suppose $\ell$ is a pencil of quadrics containing a plane, say $b=d=0$, 
and having no common singular points.  
Then by Lemma \ref{L:reducible-quartic} either $\ell=(ad-bc, be-cd)$ or $\ell=(ad-b^2, be-cd)$, up to projectivity.
However, if  $\ell=(ad-b^2, be-cd)$, then the singular points of quadrics in $\ell$ trace out the line $b=d=e=0$, 
which must then fall in the base locus of the net because $\ell$ is a double line of $\Delta(C)$.
Such a net is destabilized by the 1-PS with weights $(4,-1,4,-6,-1)$. 

If $\ell=(ad-bc, be-cd)$, then $b=d=ae-c^2=0$ is the conic along 
which elements of $\ell$ are singular, so this conic must be in the base locus of $\Lambda$. 
The claim follows.

\smallskip

(C) Suppose that we are not in the cases (A) or (B).
Then $\ell=(ae-c^2, be-d^2)$, up to projectivity. Since the generic quadric in $\ell$ has rank $4$ and the 
vertices of quadrics in $\ell$ vary along $c=d=e=0$, we deduce that $c=d=e=0$ must be in the base locus of the net.
Such $\Lambda$ is destabilized by $\rho_6=(4,4,-1,-1,-6)$.
\end{proof}

The converse to the above result is the following.
\begin{prop}
\label{P:double-conic}
A semi-stable curve with a double conic component is projectively equivalent to the intersection 
of the quadrics $Q_1=ad-bc, Q_2=be-cd, Q_3=ae-c^2+bL_1+dL_2$, where $L_i$ are not simultaneously zero.
\end{prop}

\begin{proof} Let $C$ be a semi-stable curve with a double structure supported on a smooth conic $X$. 
Denote by $\Lambda$ the associated net of $C$.
Let $\ell \subset \Lambda$ be the pencil of quadrics containing the plane $P_X$ spanned by $X$. 
Note that any quadric not in $\ell$ intersects $P_X$ in $X$, and so is smooth along $X$.
Since every point of $X$ is a singular point of $C$, for every point of $X$
there must be a quadric in $\ell$ singular at that point. 
It follows that $X$ is traced out by vertices of quadrics in $\ell$. 
(We note in passing that this implies that $\ell$ appears with 
multiplicity $2$ in $\Delta$.)
By Lemma \ref{L:reducible-quartic}, we must have $\ell=(ad-bc, be-cd)$. 
Then the singular points of quadrics in $\ell$ trace out the conic $b=d=ae-c^2=0$. 
It follows that $\Lambda=(ad-bc, be-cd, ae-c^2+bL_1+dL_2)$.
\end{proof}

\begin{remark} \label{R:double-conic} 
In Proposition \ref{P:double-conic}, the intersection of $Q_1$ and $Q_2$ is the union 
of the plane $b=d=0$ and the cubic scroll $(ad-bc, be-cd, ae-c^2)$. 
The scroll is the blow-up of $\PP^2$ at a point and is embedded in $\PP^4$ by $2H-E$, where $H$ is
the class of a line and $E$ is the class of the $(-1)$-curve.
The scroll meets the plane $b=d=0$ in the conic $ae-c^2=0$. 
The quadric $Q_3$ intersects the scroll in this conic (of class $H$ on the scroll) 
and in a curve of class $3H-2E$, which is a rational normal quartic in $\PP^4$ 
meeting the conic $b=d=ae-c^2=0$ in three points. Thus any curve described by Proposition \ref{P:double-conic} looks
like a double conic meeting a rational normal quartic in three points. 
\end{remark}

\begin{prop}\label{P:discriminant-double-conic}
The discriminant quintic of a semi-stable curve $C$ has a double conic only in the following cases:
\begin{enumerate}
\item Up to projectivity, the net of quadrics containing $C$ is $$I_R:=(ac-b^2, ae-2bd+c^2, ce-d^2).$$
\item Up to projectivity, the net of quadrics containing $C$ is $$I_{DL}:=(ad, ae+bd-c^2, be).$$
\item The curve $C$ has a double line meeting the residual genus $2$ 
curve in $2$ points; such $C$ isotrivially degenerates to the curve defined by $I_{DL}$.
\item The curve $C$ contains a double twisted cubic and is defined by Equation \eqref{E:dtc}.
\end{enumerate}
\end{prop}

\begin{proof} 
Let $C$ be a curve with a discriminant
$\Delta(C)=2q+\ell$, where $q$ is a conic and $\ell$ is the residual pencil.

\smallskip

{\em Case 1:} Suppose first that the generic quadric in $q$ has rank $3$.
The further analysis breaks into the cases enumerated
by Lemma \ref{L:singular-pencil}:
\begin{enumerate}
\item[(A)] All quadrics in $\ell$ have a common singular point; or
\item[(B)] All quadrics in $\ell$ contain a common plane; or 
\item[(C)] $\ell=(ac-b^2, ce-d^2)$, defining the Veronese quartic $V$.
\end{enumerate}
\noindent
\textbf{Observation:} Before proceeding with a case-by-case analysis, we make an elementary observation about 
conics in $\Lambda$. Namely, suppose $q \subset \Lambda$ is a conic. Suppose $[Q_1]$ and
$[Q_2]$ are two distinct point on $q$ and let $[Q_3]\in \Lambda$ be the point of the intersection 
of the tangent lines to $q$ at $[Q_1]$ and $[Q_2]$. Then $q$ can be explicitly 
parameterized by $\PP^1$ via $[s:t] \mapsto [s^2 Q_1+t^2Q_2+stQ_3]$. We also note that $Q_1, Q_2$, and $Q_3$ obtained 
by this construction span the net.

\smallskip

Case (A): The quadrics in $\ell$ have a common singular point $b=c=d=e=0$. 
We can assume that $\ell=(Q_2,Q_3)$ and 
$Q_1=a^2+\overline Q_1(b,c,d,e)$. Note that with this choice of coordinates, every quadric in the net 
can be written as a linear combination of $a^2$ and a quadric in variables $b,c,d,e$. Since elements
of $q$ have generic rank $3$, all elements of $q$ have form $\lambda a^2+R(b,c,d,e)$, where 
$\lambda\in \CC$ and 
$R(b,c,d,e)$ are generically rank $2$ quadrics with no common singular point. 
In particular, it follows that quadrics corresponding to $q\cap \ell$ are rank $2$ quadrics in variables
$b,c,d,e$. 

Given any two points $[Q_1], [Q_2] \in q$, we can choose coordinates $b,c,d,e$ so that
$Q_1=bc \pmod{a^2}$ and $Q_2=de \pmod{a^2}$. 
By the observation above every element of $q$ can be written as $s^2 Q_1+t^2Q_2+stQ_3$, 
where $[s:t]\in \PP^1$, and where $[Q_3]$ is the point of intersection of the tangent lines
to $q$ at $[Q_1]$ and $[Q_2]$. Set $\overline{Q}_3:=Q_3 \pmod a$.
Then $s^2 bc+t^2 de+st \overline{Q}_3$ is a rank $2$ quadric for all $[s:t]\in \PP^1$.

\begin{claim}\label{matrix-claim}
$\overline{Q}_3=\mu be+\frac{1}{\mu}cd$, possibly after an appropriate renaming of variables. 
\end{claim}
\begin{proof} Let $M$ be the symmetric matrix associated to $s^2 bc+t^2 de+st \overline{Q}_3$.
Analyzing $t^5s$ and $s^5t$ terms of the $3\times 3$ minors of $M$, 
we see that $\overline{Q}_3\in (bd,be,cd,ce)$. Writing $\overline{Q}_3=x bd+y be+zcd+wce$, the upper-left
$3\times 3$ minor of $M$ is $2xz s^4t^2$. Hence $xz=0$. Similarly,
one shows that $xy=zw=yw=0$. Without loss of generality, we can assume $x=w=0$. Computing 
the remaining $3\times 3$ minors we obtain $y(1-yz)=z(1-yz)=0$. The claim follows.
\end{proof}

We now consider separately two subcases:

Case (A.1): $\ell$ is tangent to $q$. 
We take $[Q_1]$ to be the point of tangency and $[Q_2]$ to be any point on the conic.  
By above, we can assume that $Q_1=de$, $Q_2=bc+a^2$, and $Q_3=be+cd$. 
(The last equality comes from $\overline{Q}_3=be+cd$ and $Q_3\in \ell$).
This net is destabilized by $\rho_2=(2,2,2,-3,-3)$.

Case (A.2): $\ell$ is not tangent to $q$. Let $q\cap \ell=\{[Q_1], [Q_2]\}$.
As above, we let $[Q_3]$ be the point of
intersection of tangents to $q$ at $[Q_1]$ and $[Q_2]$. Letting 
$Q_1=bc$ and $Q_2=de$, we can write $Q_3=\mu be+\frac{1}{\mu}cd+a^2$ by Claim \ref{matrix-claim}. 
By appropriately renaming and scaling the variables, we obtain the net $I_{DL}$.

\medskip

Case (B): The quadrics in $\ell$ contain a plane $P$ and have no common singular points.
If $\ell$ intersects $q$ in two distinct points, then we are in case (A). 
(Indeed, the quadrics in $q\cap \ell$ have rank $3$ and any two rank $3$ quadrics 
containing a common plane have a common singular point.)

Suppose $\ell$ is tangent to $q$. 
Let $[Q_1]$ be the point of tangency. Then $Q_1$ has rank at most $3$ and contains the plane
$P$. We can choose coordinates so that $P$ has equation $d=e=0$ and $Q_1=ce+\lambda d^2$, for some
$\lambda\in \CC$. Let $[Q_2]\in q\smallsetminus \ell$ and let $[Q_3]\in \ell$ 
be the point of intersection of the tangents
to $q$ at $[Q_1]$ and $[Q_2]$ (note that the tangent line to $q$ at $[Q_1]$ is $\ell$). 
Then $Q_3=dL_1+eL_2$, for some
linear forms $L_1, L_2$. Since $Q_1$ and $Q_3$ have no common singular points, the 
system of equations $c=d=e=L_1=L_2=0$ has no solutions. Thus we can 
we can assume $L_1=b$ and $L_2=a$. 

All quadrics in $q$ have form $s^2Q_1+t^2Q_2+stQ_3$, where $[s:t]\in \PP^1$ and have 
rank $3$ by our assumption. The analysis of $t^7s$ terms in the vanishing $4\times 4$ 
minors of the symmetric matrix of $s^2Q_1+t^2Q_2+stQ_3$ shows that $Q_2\in (c,d,e)$. 

It follows that $\Lambda=(Q_1,Q_2,Q_3)$, where $Q_1=ce+\lambda d^2, Q_2=ae+bd$, and $Q_3\in (c,d,e)$.
Thus $\Lambda$ is destabilized by $\rho_3=(3,3,-2,-2,-2)$.

\medskip 

Case (C): Suppose $\ell=(ac-b^2, ce-d^2)$.  
By the observation above, there exists $Q_3$ such that 
$s^2(ac-b^2)+t^2(ce-d^2)+stQ_3$
has rank $3$ for $s$ and $t$. The vanishing of the $4\times 4$ minors of 
the symmetric matrix of $s^2(ac-b^2)+t^2(ce-d^2)+stQ_3$ implies that
$Q_3=ae-2bd+c^2$. Hence the net is $I_R=(ac-b^2, ae-2bd+c^2, ce-d^2)$.

\smallskip

{\em Case 2:} We now consider the case when the generic quadric in $q$ has rank $4$. 
Taking a general pencil in $\Lambda$, we see that $C$ lies on the del Pezzo $P(1,2,2)$. 
Furthermore, by Lemma \ref{L:rank-4-singular} 
the vertices of quadrics in $q$ form a curve $D$ in the base locus of the net. 
In particular, $C$ is non-reduced. 

We now proceed to consider different cases according to the degree of $D$. 

$\bullet$ If $D$ is a line, then we are done by Proposition \ref{P:double-line} below, 
which describes all semi-stable curves containing a double line on $P(1,2,2)$.

$\bullet$ If $D$ is a conic, then we are done by Proposition \ref{P:double-conic}, 
which describes all semi-stable curves containing
double conics.

$\bullet$ 
If $D$ is a twisted cubic, then the residual component of $C$ is a conic.
Let $\ell \subset \Lambda$ be the pencil of quadrics containing the plane spanned by this conic.
The intersection of the quadrics in $\ell$ is a union of a plane and a cubic scroll. 
The general such $\ell$ has equation $(ad-bc, be-cd)$ (cf. Lemma \ref{L:reducible-quartic}) with the scroll
being $(ad-bc, be-cd, ae-c^2)$. 
Since every double twisted cubic on a scroll is a double hyperplane section, we conclude 
that every semi-stable curve with a double twisted cubic component is a degeneration of a curve 
defined by Equation \eqref{E:dtc}.

$\bullet$ If $D$ is a quartic, then by Theorem \ref{T:Description} (1) $D$ must be a rational normal quartic. 
We now consider the possibilities for $\ell$ enumerated in Lemma \ref{L:singular-pencil}. Case (B) is clearly impossible. 
In Case (A), all quadrics in $\ell$ are singular at some point $O$. 
In this case, $D$ lies on a cone over $E \subset \PP^3$ with a vertex at $O$, where $E$ is 
a complete intersection of two quadrics in $\PP^3$.
Since the arithmetic genus of $E$ is $1$ and $E$ is a projection of a rational normal quartic, 
we see that $E$ is singular.  
We conclude that $O$ lies on a chord (or a tangent line) of $D$. This immediately leads to a contradiction, as 
$D$ is a complete intersection of three quadrics not passing through $O$. 

Finally, in Case (C) $D$ lies on the Veronese $(ac-b^2, ce-d^2)$. 
By Lemma \ref{L:veronese} the ideal of $C$ is $I_R$. This finishes the proof. 
\end{proof}

\subsubsection{Double lines} 

Let $\pi\co S\ra \PP^2$ be the blow up of a plane at five points $\{p, q_1, r_1, q_2, r_2\}$, with $r_i$  
infinitesimally close to $q_i$ for $i=1,2$. 
Set $E_0:=\pi^{-1}(0)$ and let $\pi^{-1}(q_i)=F_i\cup G_i$, where $F_{1}$, $F_{2}$ are the $(-2)$-curves 
and $G_{1}$, $G_{2}$ are the $(-1)$-curves on $S$. Denote by $H$ the class of a line on $S$.
If $\phi\co S\ra \PP^4$ is the anti-canonical map, then 
$\phi(S)$ is the del Pezzo $P_1$ described in Section \ref{S:del-pezzos}. 

\begin{prop}
\label{P:double-line}
Let $C$ be a semi-stable curve on $P_1$ with a double line component. 
Then $C$ isotrivially degenerates to the curve defined by
$$I_{DL}=(ad, ae+bd-c^2, be).$$
\end{prop}

\begin{proof} 
Suppose $2L+R\in \vert -2K_S\vert$ is a divisor on $S$ such that $\phi(2L+R)$ is a semi-stable curve and $\phi(L)$ is a line.
Then $L$ is a $(-1)$-curve and $R$ meets $L$ in four points, counting multiplicities. 
But by Theorem \ref{T:Description} (3) $\phi(L)$ cannot meet $\phi(R)$ in four or more points. 
It follows that one of the irreducible components of $R$ is a $(-2)$-curve meeting $L$. 
The only $(-1)$-curves meeting $(-2)$-curves on $S$ are, up to symmetries:
\begin{enumerate}
\item[(i)] $L=G_1$ meeting the $(-2)$ curve $F_1$.
\item[(ii)] $L=H-F_1-G_1-F_2-G_2$ meeting both $(-2)$-curves $F_1$ and $F_2$.
\end{enumerate}
However, in case (i), we see that $R-F_1$ has arithmetic genus one and meets 
$L+F_1$ in three points. It follows that $\phi(R)=\phi(R-F_1)$ meets the double line $\phi(L)=\phi(L+F_1)$ in three points. 
It follows that $C$ is unstable by Theorem \ref{T:Description} (6)(b).

In case (ii), we see that $R-F_1-F_2=4H-2E_0-F_1-2G_1-F_2-2G_2$ is a genus two curve 
meeting $L+F_1+F_2$ in $2$ points.
Recall that there is an isotrivial degeneration of $P_1$ to $P_0$. 
Under this degeneration, $L$ is fixed and the limit of the residual genus two component is 
\begin{equation*}
\underbrace{(H-F_1-2G_1-E_0)}_{\text{(-2)-curve}}+\underbrace{(H-F_2-2G_2-E_0)}_{\text{(-2)-curve}}
+2E_0+(H-E_0)+(H-E_0).
\end{equation*}
It follows that under the degeneration of $P_1$ to $P_0$, $C$ is a union of two double lines of class $E_0$ and $H-F_1-G_1-F_2-G_2$, respectively,
and two conics of class $H-E_0$.
A simple computation shows that in appropriate coordinates $C$ is given by the ideal $I_{DL}=(ad, ae+bd-c^2, be)$.
\end{proof}

We summarize the discussion of this section in the following result.
\begin{prop}
\label{P:reduced-discriminant-2}
Suppose $C$ is a semi-stable curve. If $\Delta(C)$ is non-reduced, then $C$ is non-reduced.
\end{prop}
\begin{proof} $\Delta(C)$ is non-reduced if and only if it has a double line or a double conic. 
The result now follows immediately from Propositions \ref{P:discriminant-double-line}
and \ref{P:discriminant-double-conic}.
\end{proof}
\begin{corollary}\label{C:smooth-DP}
A reduced semi-stable curve lies on a smooth quartic del Pezzo.
\end{corollary}
\begin{proof}
By Proposition \ref{P:reduced-discriminant-2}, the discriminant
of a reduced semi-stable curve is reduced and so the curve lies on a smooth quartic 
del Pezzo by Lemma \ref{L:equivalence}.
\end{proof}

Conversely, we now prove Part (1) of Theorem \ref{T:reduced-semi-stable} stating that a
semi-stable curve on a smooth quartic del Pezzo is necessarily reduced.
\begin{prop}
\label{P:reduced-discriminant-1}
Suppose $C$ is a semi-stable curve. If $C$ lies on a smooth quartic del Pezzo, 
then $C$ is reduced.
\end{prop}
\begin{proof} Suppose $C$ lies on a smooth quartic del Pezzo $S$ and 
let $\ell$ be the pencil of quadrics containing $S$. 
To prove that $C$ is reduced, we argue by contradiction. 
Suppose that $C$ has a non-reduced irreducible component $D$. Then for every point $p\in D$, there is
a quadric in the ideal of $C$ that is singular at $p$. Since quadrics in $\ell$ cannot be
singular along $D$, we deduce that the ideal of $C$ contains a single quadric that is singular along all of $D$.
This leads to a contradiction using Lemma \ref{L:Weights1}. Indeed,
if $D$ is a line, then $C$ is destabilized by $\rho_3=(3,3,-2,-2,-2)$;
if $D$ is a conic, then $C$ is destabilized by $\rho_2=(2,2,2,-3,-3)$;
if $D$ is a twisted cubic, then $C$ is destabilized by $\rho_1=(1,1,1,1,-4)$.
\end{proof}

\subsection{Proofs of Theorem \ref{T:reduced-semi-stable} and \ref{T:non-reduced-semi-stable}}
\begin{proof}[Proof of Theorem \ref{T:reduced-semi-stable}]
To finish the proof of Theorem \ref{T:reduced-semi-stable}, we observe that Parts (2)--(7) 
follow from Theorem \ref{T:Description} (6)(a), (7), (8), (10), (11), (12), respectively.
\end{proof}

\begin{proof}[Proof of Theorem \ref{T:non-reduced-semi-stable}] 
To prove Theorem \ref{T:non-reduced-semi-stable}, we note that a non-reduced semi-stable curve 
has a non-reduced discriminant by Proposition \ref{P:reduced-discriminant-1} and Lemma \ref{L:equivalence}. 
It follows from Propositions \ref{P:discriminant-double-line} and \ref{P:discriminant-double-conic} that 
a non-reduced semi-stable curve degenerates isotrivially to a curve in $N_1\cup N_2\cup N_3\cup N_4$.

We note that $N_1$ and $N_4$ each consist of a single point, 
hence irreducible. To prove irreducibility of $N_2$ and $N_3$, we recall that a 
semi-stable curve with a double twisted cubic component is, up to projectivity, given by Equation \eqref{E:dtc}:
\begin{equation*}
(ad-bc, ae-c^2+L^2, be-cd).
\end{equation*} Similarly, 
 a semi-stable with a double conic component is, up to projectivity, given by Equation \eqref{E:double-conic}:
 \begin{equation*}
 (ad-bc, ae-c^2+bL_1+dL_2, be-cd), 
 \end{equation*} Irreducibility of the space of linear forms implies irreducibility of $N_2$ and $N_3$.
  
We proceed to prove the semi-stability of the general point of $N_i$ ($i=1,\dots, 4$) 
using the Kempf-Morrison criterion \cite[Proposition 2.4]{AFS}.
\begin{prop}\label{P:kempf}
The ideals $I_R, I_{DT}, I_{T}, I_{DL}$, defined by
\begin{align*}
I_R&:=(ac-b^2, ae-2bd+c^2, ce-d^2),\\
I_{DT}&:=(ad-b^2, ae-bd+c^2, be-d^2),\\
I_T&:=(ad-bc, ae+bd-c^2, be-cd),\\
I_{DL}&:=(ad, ae+bd-c^2, be),
\end{align*}
are semi-stable.
\end{prop}
\begin{proof}
Each of the ideals is stabilized by a certain 1-PS acting diagonally with respect to the distinguished basis 
$\{a,b,c,d,e\}$. Indeed, they are stabilized by $(2,1,0,-1,-2)$, $(3,1,0,-1,-3)$, $(2,1,0,-1,-2)$, and 
$(2,1,0,-1,-2)$, respectively.
By the Kempf-Morrison criterion \cite[Proposition 2.4]{AFS}, 
it therefore suffices to check that these curves are semi-stable 
with respect to 1-PS's acting diagonally with respect to this basis. 
By Theorem \ref{T:subgroups} (see also Remark \ref{R:torus-remark}), 
it suffices to check the given finite list of 1-PS's acting diagonally with respect to this basis, 
and this is an easy exercise.
We should remark that semi-stability of the balanced ribbon $I_R$ 
is a special case of a more general \cite[Theorem 4.1]{AFS}. 
\end{proof}
Observing that $I_R \in N_1, I_{DT}\in N_2, I_T\in N_3, I_{DL} \in N_4$ finishes the proof of 
Theorem \ref{T:non-reduced-semi-stable}.
\end{proof}
\begin{remark}\label{R:triple-conic-rank-3-quadrics}
We note that $I_T=(ad-bc, ae+bd-c^2, be-cd)$ contains no rank $3$ quadric. Thus
\cite[Theorem 2.1]{fedorchuk-jensen} provides an easier, independent proof for the semi-stability of $I_T$.
\end{remark}
\begin{remark}\label{R:triple-conic} 
We observe that the discriminant quintic of $I_T$ is $2y^3(xz-y^2)$, which is unstable under the natural
$\SL(3)$-action on the space of plane quintics.
\end{remark}

\section{Proofs of the main results}\label{S:proofs}
In this section we prove Main Theorems 1--3 from the introduction. 
Main Theorem \ref{MT1} follows immediately from Theorems 
\ref{T:reduced-semi-stable} and \ref{T:non-reduced-semi-stable}.

\begin{proof}[Proof of Main Theorem \ref{MT2}]
The divisor of singular complete intersections in $\MX$ is irreducible. Furthermore, from Main Theorem \ref{MT1}, its
general point corresponds to a one-nodal curve on a smooth quartic del Pezzo. It follows that the 
inverse rational map $f^{-1}\co \MX \dashrightarrow \M_5$ does not contract divisors and so $f$ is a contraction. 

By Theorem \cite[Theorem 6.2 and Theorem 6.3]{hassett-stable},
the general curve in $\Delta_1$ arises as a stable limit of a genus $5$ curve with
an $A_2$ singularity and the general curve in $\Delta_2$ arises as a stable limit of the general curve
in the $A_5^{\{1\}}$-locus. 
It follows that $\Delta_1$ and $\Delta_2$ 
are generically fibered over $A_2$- and $A_5^{\{1\}}$-loci, respectively.

It remains to prove that the trigonal divisor is contracted to a single point given by Equation \eqref{E:triple-conic}.
Recall that a general smooth trigonal curve of genus $5$ has a very ample canonical line bundle 
and its canonical embedding lies on a cubic scroll, whose homogeneous ideal is 
$(ad-bc, ae-c^2, be-cd)$, up to projectivities. 
A trigonal curve on the scroll is cut out by two linear independent cubics 
$$(aR_1-bR_2+cR_3, cR_1-dR_2+eR_3),$$ where $R_1,R_2,R_3$ are quadrics in $\CC[a,b,c,d,e]$.
In particular, a general trigonal curve is obtained by taking $\{R_i\}_{i=1}^{3}$ to be general.

\begin{prop}\label{P:trigonal-contraction} 
The rational map $f\co \M_5 \dashrightarrow \MX$ 
contracts the trigonal divisor $\Trig_{5}$ to the point 
$$I_T:=(ad-bc, ae+bd-c^2, be-cd)\in \MX.$$
\end{prop}

\begin{proof}
Since $\Trig_5$ is a divisor, $f$ is defined at the generic point of $\Trig_5$. Hence to show that $f$ contracts
$\Trig_5$ to a point, we need to show that a general trigonal curve arises as a stable limit for
some deformation of $I_T$. 

Consider the family of nets $\Lambda_t=(Q_1(t), Q_2(t), Q_3(t))$ defined by 
\begin{align*}
Q_1(t) &=be-cd+t^2R_1(a,b,c,d,e), \\
Q_2(t) &=ae-c^2+t^2R_2(a,b,c,d,e), \\
Q_3(t) &=ad-bc+t^2R_3(a,b,c,d,e),
\end{align*}
where $R_1,R_2,R_3$ are general quadrics. 
Then for $t\neq 0$, $\Lambda_t$ defines a smooth non-trigonal curve $C_t$ of genus $5$,
while $\Lambda_0=(ad-bc, ae-c^2, be-cd)$ defines a cubic scroll in $\PP^4$.
Let $C_0$ be the flat limit of $\{C_t\}_{t\neq 0}$ as $t\to 0$. 
Using linear syzygies among the quadrics containing the scroll, we see that 
\begin{align*}
F_1 &:=\frac{1}{t^2}\bigl(aQ_1(t)-bQ_2(t)+cQ_3(t)\bigr)\vert_{t=0}=a R_1-b R_2+c R_3, \\
F_2 &:=\frac{1}{t^2}\bigl(cQ_1(t)-dQ_2(t)+eQ_3(t)\bigr)\vert_{t=0}=cR_1-d R_2+eR_3
\end{align*}
are cubics in the ideal of $C_0$. Since $R_1, R_2, R_3$ were chosen generically, we conclude
that $C_0$ is a general smooth trigonal curve. 

It now suffices to show that the limit of $\{\Lambda_t\}_{t\neq 0}$ in $\MX$ is the point $I_T$.
Let $\rho_t \in \operatorname{PGL}(5)$ be given by 
$\rho_t\cdot [a:b:c:d:e]=[a:t^{-1}b: c: t^{-1}d: e].$
Set $\Lambda'_t:=\rho_t\cdot \Lambda_t$.
Then 
the flat limit of $\Lambda'_t$ as $t\to 0$ is 
$\Lambda'_0=(ad-bc, ae-c^2+R_2(0,b,0,d,0), be-cd)$. 
Since $R_2$ was chosen to be a general quadric, $S(b,d):=R_2(0,b,0,d,0)$ has rank $2$.

We claim that, without loss of generality, we may take 
$R_2(0,b,0,d,0)=bd +\eta d^2$ for some scalar $\eta$. 
This implies that $\Lambda'_0$ is semi-stable and that its orbit closure contains $I_T$, 
since the limit  as $t\to \infty$ of 
$(ae-bc, ae-c^2+ bd +\eta d^2, be-cd)$ 
under the one-parameter subgroup $(t^2,t,1,t^{-1},t^{-2})$ is 
$I_T=(ad-bc, ae-c^2+ bd, be-cd)$.

It remains to show that we may take $R_2(0,b,0,d,0)=bd +\eta d^2$. 
Let $S(b,d)=L_1(b,d)L_2(b,d)$, where $L_1$ and $L_2$ are linearly independent linear forms.
Without loss of generality, $L_1(b,d)=d+\mu b$, where $\mu\in C$.
Make the following coordinate change:
\begin{align*}
 a' &:=a, \\
 b' &:=b, \\
 c' &:=c+\mu a, \\
 d' &:=d+\mu b, \\
 e' &:=e+2\mu c +\mu^2 a. 
 \end{align*}
 Let $M(b',d')=\lambda b'+\nu d'$ be the linear form such that $M(b',d')=L_1(b,d)$. 
 Note that $\lambda\neq 0$. After scaling, we can assume that $\lambda=1$.
 Then 
 \begin{align*}
 a'd'-b'c' &=ad-bc, \\ 
  b'e'-c'd' &=be-cd-\mu(ad-bc),     \\ 
 a'e'+M(b',d')d'-(c')^2&= ae-c^2+L_1(b,d)L_2(b,d),
\end{align*}
as desired.
\end{proof}
\renewcommand{\qedsymbol}{}
\end{proof}

\begin{proof}[Proof of Main Theorem \ref{T:moving-slope}]
By Main Theorem \ref{MT2}, $f$ is a contraction. We compute $f^*\O(1)$ using two methods. 

The most straightforward way is to write 
down three test families along which $f$ is regular and which are contracted by $f$.
Consider the following families in $\Mg{5}$:
\begin{enumerate}
\item A family $T_1$ of elliptic tails attached to a fixed general pointed genus $4$ curve.
We have $\lambda\cdot T_1=1$, 
$\delta_0\cdot T_1=12$, $\delta_1\cdot T_1=-1$, $\delta_2\cdot T_1=0$.
Furthermore, deformations of $T_1$ cover $\Delta_1$.

\item A family $T_2$ of genus $2$ tails attached to a fixed general pointed genus $3$ curve at 
a non-Weierstrass point; see \cite[Section 4.4]{handbook} for a precise description of the construction. 
We have $\lambda\cdot T_3=3$, $\delta_0\cdot T_3=30$, $\delta_2\cdot T_3=-1$, $\delta_1\cdot T_3=0$.
Furthermore, deformations of $T_2$ cover $\Delta_2$.
\item A family $T_3$ of curves in $\Trig^{0}$ satisfying $\lambda\cdot T_3=4$, 
$\delta_0\cdot T_3=33$, $\delta_i\cdot T_3=0$ for $i=1,2$. Such a family exists by \cite{anands},
where it is also shown that deformations of $T_3$ cover $\Trig_5$.
\end{enumerate}

By Main Theorem \ref{MT2}, $f$ contracts each $T_i$. Namely, $f(T_1)$ is a semi-stable cuspidal curve,
$f(T_2)$ is a semi-stable curve in the $A_{5}^{1}$-locus, and $f(T_3)=[I_T]$.
 Therefore, {\em assuming that $f$ is regular along each $T_i$}, 
we have $f^*\O(1).T_i=0$ for each $i=1,2,3$. 
Writing $f^*\O(1)=a\lambda-b\delta_0-c\delta_1-d\delta_2$ and intersecting both sides with $T_i$, we obtain
\[
f^*\O(1) \sim 33\lambda-4\delta_0-15\delta_1-21\delta_2
\]
as desired.
Unfortunately, proving $f$ is regular along each $T_i$ directly would require 
several rather subtle stable reduction calculations. 
Thus, we give an alternative computation of $f^*\O(1)$, 
which implies the desired regularity {\em a posteriori}. 

Let $\mathcal{D} \subset \MX$ be the divisor of nets containing rank $3$ quadrics. 
Note that, $I_R \in \mathcal{D}$ (see Proposition \ref{P:kempf} for the definition of $I_R$). 
In particular, $\mathcal{D}$
is non-empty. $\mathcal{D}$ is irreducible because the divisor of nets in $\Gr(3,15)$
containing a rank $3$ quadric is irreducible. By Remark \ref{R:triple-conic-rank-3-quadrics},
$I_T\notin \mathcal{D}$, and since $I_T$ lies in the closure of $A_2$- and $A_5^{\{1\}}$-loci, we conclude
that $f^{-1}\mathcal{D}$ does not contain $\Delta_1$, $\Delta_2$, or $\Trig_5$. It follows that
$f^{-1}\mathcal{D}$ is the divisor of genus $5$ curves with a vanishing theta-null. 
By \cite[Proposition 3.1]{teixidor}, the class of this divisor is proportional to 
$4(33\lambda-4\delta_0-15\delta_1-21\delta_2)$.
This proves Part (1) of the theorem.

To prove Part (2), simply observe that 
\[
K_{\Mg{5}}+\frac{14}{33}\delta 
\sim \bigl(33\lambda-4\delta_0-15\delta_1-21\delta_2\bigr)+11\delta_1+17\delta_2
\]
and 
\[
K_{\Mg{5}}+\frac{3}{8}\delta 
\sim \bigl(8\lambda-\delta_0-4\delta_1-6\delta_2\bigr)+3\delta_1+5\delta_2,
\]
where $8\lambda-\delta_0-4\delta_1-6\delta_2$ is an effective multiple of the divisor class of $\Trig_5$ 
by Brill-Noether Ray Theorem \cite[Theorem 6.62]{HarMor}.
\end{proof}

\bibliographystyle{amsalpha}

\end{document}